\documentclass[11pt,a4paper]{article}
\usepackage{amssymb}
\usepackage{amsthm}
\usepackage{amsmath}
\usepackage[left=2cm,right=2cm]{geometry}

\newtheorem{theo}{Theorem}
\newtheorem{prop}{Proposition}
\newtheorem{coro}{Corollary}

\newtheorem{lema}{Lemma}
\newtheorem{defi}{Definition}
\def\be#1\ee{\begin{equation}#1\end{equation}}
\newcommand{\ba}{\begin{eqnarray} }
\newcommand{\ea}{\end{eqnarray} }
\def\bt#1\et{\begin{theo}#1\end{theo}}
\def\bl#1\el{\begin{lema}#1\end{lema}}
\def\bp#1\ep{\begin{prop}#1\end{prop}}
\def\bd#1\ed{\begin{defi}#1\end{defi}}

\def\cca{{\bf a}}
\def\ccb{{\bf b}}
\def\ccc{{\bf c}}
\def\ccA{{\cal A}}
\def\ccB{{\cal B}}
\def\ccC{{\cal C}}
\def\ccD{{\cal D}}

\def\ccF{{\cal F}}

\def\ccX{{\cal X}}
\def\va{\varepsilon}
\def\ra{\rightarrow}

\def\E{\mathbf{E}}
\def\P{\mathbf{P}}
\def\Var{\mathbf{Var}}

\def\N{{\mathbb N}}

\def\R{{\mathbb R}}
\def\Z{{\mathbb Z}}

\def\ls{\le}
\def\gs{\ge}

\title{\bf Exponential Concentration Inequalities for Additive Functionals of Markov Chains}
\author{Rados{\l}aw Adamczak\thanks{Research partially supported by  MNiSW Grant N N201 608740 and the Foundation for Polish Science} and Witold Bednorz\thanks{Research partially supported by  MNiSW Grant N N201 608740}\\University of Warsaw}

\begin{document}

\maketitle
\begin{abstract}
Using the renewal approach we prove exponential inequalities for
additive functionals and empirical processes of ergodic Markov
chains, thus obtaining counterparts of inequalities for sums of
independent random variables. The inequalities do not require
functions of the chain to be bounded and moreover all the
involved constants are given by explicit formulas whenever the usual drift condition holds,
which may be of interest in practical applications e.g. to MCMC
algorithms.
\\
AMS 2000 subject classifications : Primary 60E15, 60J20, secondary 60K05, 65C05\\
Keywords: Markov chains, exponential inequalities, drift criteria
\end{abstract}

\section{Introduction}\label{sect0}

This paper concerns exponential type concentration inequalities
for additive functionals of Markov chains, i.e. for sums of the
form
\begin{displaymath}
f(X_0)+\ldots+f(X_{n-1}),
\end{displaymath}
where $(X_i)_{i\in \N}$ is a Markov chain.

Such inequalities are widely used in Markov Chain Monte Carlo
theory to provide estimates on the rate of convergence for certain
algorithms. Moreover, in certain statistical applications (e.g. for
M-estimators) one is interested in estimates on suprema of such
functionals over some classes $\mathcal{F}$ of functions.

Concentration phenomenon in the Markov chain setting has been
studied in many papers to mention
\cite{Marton,Samson,Cle,MeynKos1,MeynKos2,DGM,KontoRam,Cha,Rad,BerCle}. Clearly in general one
cannot hope to recover classical results for sums of independent
random variables at their full strength. Therefore the goal is
to provide counterparts of the inequalities for independent
summands under conditions, which are relatively easy to verify and
involve only 'computable' characteristics of the chain. From the
practical point of view, it is also important to derive estimates
with explicit and 'reasonable' constants.

Among the most successful approaches developed to obtain deviation
inequalities for Markov chains one can list the transportation of
measure method (see \cite{Marton, Samson}), related coupling techniques \cite{Cha}, martingale
approximation \cite{KontoRam} and renewal theory \cite{Cle,DGM,Rad} based on the splitting technique introduced by Nummelin \cite{NumSp} and Athreya and Ney \cite{AN} (for general
introduction see \cite{MeynTwe,Num1,Che}).
\smallskip

\noindent In this paper we follow the renewal approach, i.e. we
decompose the sum of functionals of a Markov chain into excursions
from and to an existing or artificially created atom. This concept
allows to reduce the question of exponential inequalities for
Markov chains to the concentration of sums of independent (or
nearly independent) random variables at the expense of some
further technical work. Moreover, under additional assumptions, it
allows to incorporate in the estimates the limiting variance of a
rescaled functional and thus to obtain Bernstein type inequalities
corresponding to the central limit theorem. For reader's
convenience and further use we recall below the classical
formulation for independent summands (see e.g. \cite{Bou})

\noindent {\bf Bernstein's inequality} {\it Let
$(\xi_i)^{\infty}_{i=0}$ be independent random variables such that
$\E \xi_i=0$ and $|\xi_i|\ls M$. Let $\E \xi_i^2=\sigma^2$, then for any $t > 0$,
$$
\P(|\sum^{n-1}_{i=0} \xi_i|\gs t)\ls 2\exp(-\frac{t^2}{2(n\sigma^2+Mt/3)}).
$$}

If one insists on having an estimate in terms of the variance,
Bernstein's inequality requires that all random variables be
bounded, which is rather restrictive. Clearly one expects similar
results to hold under more general assumptions about
integrability of the summands, e.g. when
$\E\exp(|\xi_i|^{\alpha}/c^{\alpha})\ls 2$ for some $\alpha > 0$ (which corresponds to
finiteness of exponential Orlicz norms). This is indeed the case,
though the inequality is a little bit weaker (see the results by Borovkov \cite{Bor1,Bor2}). In a more general
context of empirical processes inequalities for independent summands with finite exponential Orlicz norms were proved in \cite{Rad}.

Let us now describe in more detail the setting we consider and state some of the results we obtain. We remark that for simplicity in the Introduction we present precise estimates only for a special case of geometrically ergodic Markov chains, however the results we obtain allow to deduce bounds under weaker assumptions of subexponential ergodicity. Below we indicate the form of the inequality we will obtain in the general case, postponing precise definition of the parameters involved to the main body of the article.

Let ${\bf X}=\{X_k:\; k\in
\Z_{+}\}$ be a time-homogeneous Harris ergodic Markov chain defined on the state
space $(\ccX,\ccB)$ (to avoid certain measurability issues we will
assume that $\ccB$ is countably generated, which is enough for all
the applications we have in mind) and let ${\bf X^m}$ denote its
$m$-skeleton, i.e. ${\bf X^m}=\{X_{mk}:\; k\in \Z_{+}\}$. We will denote by $(\Omega,\ccF,\P)$ the general
probability space on which the process is defined and by
$\P_{\mu}$ its conditional version where the starting distribution
equals $\mu$, i.e $\P_{\mu}(X_0\in A)=\mu(A)$, $A\in\ccB$. For
simplicity we write $\P_x$ whenever $\mu=\delta_x$. Let $\P(x,A)$,
$x\in\ccX$, $A\in \ccB$ denote the Markov chain transition
function and let $P$ be the operator on the measurable functions
given by $(P f)(x)=\E_x f(X_1)=\int f(y)\P(x,dy)$.

The general theory of Markov chains states that whenever
the chain is aperiodic and there exists an invariant probability
measure $\pi$ on $(\ccX,\ccB)$ then the small set condition is
verified, i.e. there exists $C\in\ccB$ of positive $\pi$-measure,
a probability measure $\nu$, $\nu(C)>0$, $\delta>0$ and an integer
$m$, such that \be\label{small} \P^m(x,B)\gs \delta
\nu(B),\;\;x\in C,\;\;B\in\ccB. \ee Moreover the $m$-step chain
$\{X_{km}:\;k\in \Z_{+}\}$ may be split to form a chain that
possesses a recurrent atom. The construction is well known and as
references we recommend  \cite{Num1, MeynTwe,Che}. We recall it briefly and present its basic properties in Section \ref{sect1}.

Following the general regeneration approach to Markov chains, one
can apply inequalities for independent or one-dependent unbounded random variables to excursions from and to an
existing or artificially constructed atom of a chain.

Following this strategy we provide bounds on
$\P_x(|\sum^{n-1}_{k=0}f(X_k)|>t)$ in the form close to the one of Bernstein's
inequality, i.e. for some $\alpha \in (0,1]$ and all $t \ge 0$,
\begin{align*}
 \P_x(|\sum^{n-1}_{i=0}f(X_i)|>t)\ls& C\exp(-\frac{c_1t^{\alpha}}{\cca^{\alpha}})+ C\pi^{\ast}(\theta)^{-1}\exp(-c_2\frac{t^{\alpha}}{\ccb^{\alpha}})+
C\exp(-\frac{c_3t^{\alpha}}{2(4\ccc)^{\alpha}}) \\
&+ C
\exp(-\frac{c_4t^2}{(\lceil n/(2m)\rceil\sigma^2+Mt/6)}).
\end{align*}
with $M=\ccc(3\alpha^{-2}\log(n/m))^{1/\alpha}$, where $\cca,\ccb,\ccc$ are certain parameters of the chain and $\sigma^2 = \E s_0^2$ is the variance of the block appearing in the regeneration construction, $\theta$ is the artificially created atom of the split chain and the constants $C,c_k$ depend only on $\alpha$ (through explicit formulas). At this point, to avoid technicalities, we do not present precise definition of the parameters appearing in the above inequality or the assumptions under which it holds and mention only that in Section \ref{section_drifts} we provide tools allowing to estimate them by means of various drift conditions, which in the Markov chain setting are the most convenient and widely used technique for proving exponential integrability. In general it is known \cite{MeynKos2} that for $\alpha=1$ one cannot avoid geometric drifts that are usually
difficult to work with and verify in practice (see Section \ref{sect_milt_dift}). To avoid this obstacle we show that usual drift condition related to (sub)-geometric ergodicity also provides a bound, however with $\alpha<1$ (see Section \ref{secRegularDrift}).

The simplest setting we describe is that of geometric
ergodicity where we show a class of functions that depends on the drift condition only for which one can obtain strong concentration results.
Again we pay attention to provide explicit constants in the estimates which may allow for applications in MCMC algorithms.
Recall thus the classical drift condition, which is satisfied if for some $V\colon \mathcal{X}\to [1,\infty)$,
\begin{align}\label{eq:geometric_ergodicity_drift_new}
 (PV)(x)-V(x)\ls -\lambda V(x)+b1_C(x),
\end{align}

It is well known that the above condition is equivalent to geometric ergodicity of the chain, i.e. to exponential convergence $\|P^n(x,\cdot) - \pi\|_{TV} \le C(x)\rho^n$ for some $\rho \in (0,1)$ and all $x\in \mathcal{X}$ (see \cite{MeynTwe} for the general theory).

The following theorem (which follows from our general results) provides concentration for additive functionals of geometrically ergodic strongly aperiodic (i.e. satisfying \eqref{small} with $m=1$) Markov chains in terms of the parameters appearing in \eqref{small} and \eqref{eq:geometric_ergodicity_drift_new}. To simplify the presentation and postpone technical details we provide here a simplified version of the result, which at full strength is given in Theorem \ref{thm:main_geometric_new} in Section \ref{sect3}.

\begin{theo}\label{thm:A_new}
Assume that $\{X_n\}_{n\ge0}$ is a Harris recurrent strongly aperiodic Markov chain on the space $(\mathcal{X},\mathcal{B})$ admitting a unique stationary measure $\pi$. Assume furthermore that conditions \eqref{small} with $m=1$ and \eqref{eq:geometric_ergodicity_drift_new} are satisfied.
Let finally $s > 0$ and consider an arbitrary measurable function $g\colon \mathcal{X} \to \R$ such that
\begin{displaymath}
|g(x)| \le \kappa \Big(\log V(x)\Big)^s
\end{displaymath}
for some $\kappa \ge 0$. Set $\alpha = 1/(s+1)$. Then for every $x \in \mathcal{X}$, $\eta \in (0,1]$ and $t > 0$,
\begin{align}\label{eq:main_est_geometric_erg_new}
& \P_x(|\sum^{n-1}_{i=0}g(X_i) - n\pi g|>t)\ls
2\exp(-\frac{(t\eta)^{\alpha}}{A_1^{\alpha}})+
C\exp(-\frac{(t\eta)^{\alpha}}{A_2^{\alpha}})\\
&+ e^{8}\exp(-\frac{(\eta t)^{\alpha}}{A_3^\alpha}) +
2^{1+\eta/(2+\eta)}\exp(-\frac{t^2}{2((1+\eta)\sigma^2n+A_4(\eta)(\log n)^{1/\alpha}
t)}),\nonumber
\end{align}
\noindent where
\begin{align}\label{eq:sigmadef_new}
\sigma^2 = \lim_{n\to \infty} \frac{\Var(\sum_{i=0}^{n-1} g(X_n))}{n} = \Var_\pi g(X_0)^2 + 2\sum_{i=1}^\infty {\rm Cov}_\pi g(X_0)g(X_i),
\end{align}
the constants $C,A_1,A_2,A_3$ depend only on $\delta, V, \alpha, \lambda, b, \kappa$ and the constant $A_4(\eta)$ depends only on $\delta,V,\alpha, \lambda, b, \kappa, \eta$ (in all cases the dependence is explicit, $A_4(\eta) \simeq 1/\eta$ for $\eta \simeq 0$).
\end{theo}

We also provide a counterpart of the above result for suprema of empirical processes (in the spirit of Talagrand's inequalities in the independent case).

\begin{theo}\label{thm:B_new}
Assume that $\{X_n\}_{n\ge0}$ is a Harris recurrent strongly aperiodic Markov chain on the space $(\mathcal{X},\mathcal{B})$ admitting a unique stationary measure $\pi$. Assume furthermore that conditions \eqref{small} with $m=1$ and \eqref{eq:geometric_ergodicity_drift_new} are satisfied.
Let finally $s > 0$ and consider a countable class $\mathcal{G}$ of measurable functions $g\colon \mathcal{X} \to \R$ such that for all $g \in \mathcal{G}$,
\begin{displaymath}
|g(x)| \le \kappa \Big(\log V(x)\Big)^s
\end{displaymath}
for some $\kappa \ge 0$. Set $\alpha = 1/(s+1)$.
Denote
\begin{displaymath}
Z= \sup_{g\in\mathcal{G}}|\sum_{i=0}^{n-1}(g(X_i) - \pi g)|
\end{displaymath}
and
\begin{displaymath}
\sigma^2 = \sup_{g\in\mathcal{G}} \lim_{n\to \infty} \frac{\Var(\sum_{i=0}^{n-1} g(X_n))}{n} = \sup_{g\in \mathcal{G}} \Big(\Var_\pi g(X_0) + 2\sum_{i=1}^\infty {\rm Cov}_\pi g(X_0)g(X_i)\Big).
\end{displaymath}

Then for every $\eta \in (0,1)$ and all $t > 0$,
\begin{align*}
\bar{\P}_{x^\ast}\Big(Z \ge (1+\eta) \E Z + t\Big) &\le \exp\Big(-\frac{t^2}{2(1+\eta) n\sigma^2}\Big) +
C\exp\Big(-\frac{t^\alpha}{2B_1^\alpha}\Big)
+ C\exp\Big(-\frac{t^\alpha}{B_2^\alpha}\Big)\\
&+C\exp\Big(-\frac{t^\alpha}{B_3^\alpha}\Big)
+C\exp\Big(-\frac{n}{B_4}\Big)+C\exp\Big(-\frac{t}{B_5\log n}\Big),
\end{align*}
where the constants $C, B_1,\ldots, B_5$ depend only on $\delta,V,\alpha,\lambda,b,\kappa,\eta$.
\end{theo}

\paragraph{Remarks} Since there are many concentration inequalities for additive functionals of Markov chains with all sorts of assumptions and sometimes it is difficult to compare their strength, we will now provide a few remarks concerning the above theorem as well as other results of the paper (even though their precise formulation is postponed to further sections), which should explain their strength and shortcomings.

1. As already mentioned, Theorem \ref{thm:A_new} is just a single explicit example, which can be recovered from our estimates for additive functionals. In the paper we consider a more general abstract setting in which similar exponential inequalities hold. From this setting one can obtain inequalities under other, more general drift conditions, guaranteeing subexponential convergence to the stationary distribution (such drift conditions have been considered in e.g. \cite{DFMS,DGM}). However, since the structure of the estimates on the counterparts of $A_1,\ldots,A_4$ in this inequalities may depend on the drift condition, we restrict to the most classical drift condition \eqref{eq:geometric_ergodicity_drift_new}.

2. An important feature of our results is that all the constants and parameters in our inequalities are given by explicit formulas. This allows for potential use of our inequalities in the MCMC setting to get quantitative results for various algorithm. In Sections \ref{secRegularDrift} and \ref{sect3} we illustrate it with a simple example of a Metropolis-Hastings algorithms on $\N$ and on $\R$.

Unfortunately, the dependence of constants on the parameters, especially on $\delta$, is rather bad. In certain situations one may `amplify' the parameter $\delta$ by considering a lazy version of the random walk, however in most situations the bounds we obtain by combining our inequalities with known estimates on the parameters $\delta,\lambda, b$ are not good enough to be applied in algorithmic practice. From this point of view our result may be considered the first step towards more practical estimates.

3. A fact important from the theoretical perspective is that the subgaussian coefficient in our inequality \eqref{eq:main_est_geometric_erg_new} is $(2+\eta)n\sigma^2$, where $\sigma^2$ is the variance of the limiting Gaussian variable in the CLT for
\begin{displaymath}
\frac{f(X_0)+\ldots+f(X_{n-1}) - n\pi f}{\sqrt{n}}
\end{displaymath}
(see \cite{MeynTwe,Num1,Che}).

Moreover one can easily obtain from our result a tail estimate of the form
\begin{displaymath}
\P(n^{-1/2}|\sum_{i=0}^{n-1} f(X_i) - n\pi f|\ge t) \le C\exp\Big(-\frac{t^2}{(2+\eta)\sigma^2}\Big) + C\exp\Big(-\frac{t^\alpha n^{\alpha/2}}{R_\eta a_{\alpha,f}\log n}\Big),
\end{displaymath}
where $C$ is a universal constant $R_\eta$ depends only on $\eta$ and $a_{f,\alpha}$ depends only on $f,\alpha$ and the parameters in \eqref{small} and \eqref{eq:geometric_ergodicity_drift_new}. Thus for large $n$ the estimate almost coincides with the tail of the limiting Gaussian variable. This is closely related to the moderate deviation principle (see \cite{MPR} for moderate deviation results for mixing sequences), which however is an asymptotic statement.

Similarly, the coefficient $\sigma^2$ in Theorem \ref{thm:B_new} corresponds to the weak variance of the limiting Gaussian process (provided it exists, see e.g. \cite{vdVW,BL,Rio2} for various criteria guaranteeing it) and thus corresponds to the Gaussian concentration inequality.

The fact that the constant in front of $\sigma^2$ can be made arbitrarily close to the optimal value $2$ may be important in strong limit theorems for various statistics involving additive functionals or their suprema as it is the case for independent summands.

4. Finally let us compare our results with some other estimates for additive functionals of Markov chains known in the literature.

As we already mentioned there are many results concerning exponential inequalities. A huge part of the literature is devoted to uniformly ergodic Markov chains (see e.g. \cite{Samson,KontoRam}). Methods designed to handle the uniformly ergodic case usually do not allow to deal with general geometrically or subgeometrically ergodic chains as the parameters of the inequalities grow too rapidly with the sample size (see  Section 3.5 of \cite{Rad} for a more thorough discussion).

In the Markov chains literature existing results have been mostly concerned with the case of bounded $f$. One could mention here e.g. \cite{DGM,Cle,Samson,Rad}. Our results can be seen as a generalization of \cite{Rad}. An advantage of our results with respect to the previous ones is that they work for unbounded functions $f$ and as already mentioned the dependence of constants on the parameters of the drift criterion is explicit. Also, contrary to the results in \cite{DGM,Cle} (which were based on the Fuk-Nagaev inequality), the logarithmic dependence on $n$ in our inequalities is optimal (see the discussion in \cite{Rad}).

Another approach, which is used in the Markov chain literature, mostly in the continuous time or discrete state-space case is based on spectral gap inequalities. Its advantage stems from the fact that it often allows to recover a precise Bernstein type inequality for bounded functions of the same form as in the independent case (see e.g. \cite{Lezaud,GGW}). We note however that the Markov chains we consider do not have to satisfy the spectral gap as in the non-reversible case it is not implied by geometric ergodicity (see \cite{KontoMeyn1}). Moreover, the examples in \cite{Rad} show that in the non-reversible case such an inequality cannot hold (which is a consequence of the already mentioned optimality of the factor $\log n$ in our estimates). Also, even in the reversible discrete time case, we are not aware of a Bernstein type inequality with the right subgaussian coefficient (meaning a universal constant multiplied by the limiting variance).

Other interesting results concerning concentration inequalities can be found in literature on mixing sequences (see \cite{MPR,Samson, Winte1,Winte2,Rio1}). The results obtained there usually are not directly comparable with ours as they are expressed in terms of different quantities and also often correspond rather to Azuma-Hoeffding inequalities than to the Bernstein inequality (meaning that the parameter $\sigma^2$ is not taken into account). We are aware of one exception, namely a very general result of \cite{MPR}. Since the inequality presented in \cite{MPR} works in a much more general setting than ours and is of a very similar form, we would like now to present a more detailed comparison and explained why, despite formal similarity, our results provides in certain cases additional information.

In \cite{MPR} the Authors consider a sequence of centered random variables $(X_i)_{i\in \Z}$ which is $\alpha$-mixing with the mixing coefficient $\alpha(n) \le \exp(-cn^{\gamma_1})$ (we refer to \cite{MPR} for the definition of the mixing coefficients and explain only that in the geometrically ergodic case we have $\gamma_1 = 1$) and such that $\|X_i\|_{\psi_{\gamma_2}} \le b$ for some $\gamma_2$ such that $\gamma := \frac{\gamma_1\gamma_2}{\gamma_1+\gamma_2} < 1$. For such random variables the Authors obtain an inequality of the form

\begin{align}\label{eq:MPR}
\P(\max_{j\le n}|X_1+\ldots+X_n| \ge t) \le& n\exp\Big(-\frac{t^\gamma}{C_1}\Big) + \exp\Big(-\frac{t^2}{C_2(1+nV)}\Big)\\
&+\exp\Big(-\frac{t^2}{C_3 n}\exp(\frac{t^{\gamma(1-\gamma)}}{C_4(\log t)^\gamma}\Big)\Big),\nonumber
\end{align}
where the constants $C_1,\ldots,C_4$ depend only on $b,c,\gamma_1,\gamma_2$ and
\begin{displaymath}
V = \sup_{M> 0}\sup_{i > 0}(\Var \varphi_M(X_i) + 2\sum_{j>i}|{\rm Cov}(\varphi_M(X_j),\varphi_M(X_i))|),
\end{displaymath}
with $\varphi_M(x) = (x\wedge M)\vee(-M)$.

Before we proceed with a comparison, let us explain that the parameter $\gamma$ above, although defined in a different language than $\alpha$ in Theorem \ref{thm:A_new}, agrees with it (this can be seen by looking at the general case considered in Section \ref{secRegularDrift}, where analogous formulas appear).

For comparison let us disregard the fact that the above inequality deals with maxima of partial sums (it can be obtained from a version without maxima with a recent result from \cite{KM}).

As we already mentioned, the above inequality is much more general than ours, however the sets of assumptions are slightly different, namely ours is expressed in terms of drift parameters, which are quite common in MCMC practice, while \eqref{eq:MPR} is more abstract, as it deals with general mixing sequences. From this point of view, verification of our assumptions is simpler and more straightforward. Moreover, in our result the constants are explicit. Finally, the subgaussian coefficient $V$ in \eqref{eq:MPR} is not the variance of the limiting Gaussian distribution, contrary to our $\sigma^2$ (although there are some formal similarities between the formula for $V$ and the expansion of $\sigma^2$ given in \eqref{eq:sigmadef_new}). Also the constant in front of the subgaussian coefficient in \eqref{eq:MPR} is not universal, but depends on the parameters $\gamma_i$.

Thus, while \eqref{eq:MPR} is more general and has many theoretical applications (e.g. the moderate deviation principle obtained in \cite{MPR}), our inequality seems potentially more suited e.g. for the algorithmic applications we have in mind and is more closely related to the CLT for additive functionals.

Let us also mention that using the drift criteria from Section \ref{sect_milt_dift} we can obtain an inequality with $\alpha = 1$, whereas \eqref{eq:MPR} deals with $\gamma = \alpha < 1$.

5. It would be interesting to extend Theorems \ref{thm:A_new} and \ref{thm:B_new} to Markov chains which are not necessarily strongly aperiodic (i.e. to chains which satisfy \eqref{small} only with some $m > 1$). Our Theorem \ref{thm2} provides an exponential inequality also in this setting, however the subgaussian coefficient in the general case is no longer the variance of the limiting Gaussian distribution. To our best knowledge the problem of providing a Berstein type inequality with the `right' subgaussian coefficient for $m > 1$ remains open even in the case of bounded functions.

\paragraph{}
The organization of the paper is as follows: in Section
\ref{sect1} we discuss the Markov chain theory and a variety of
integrability conditions we need to prove exponential
concentration and in Section \ref{section_drifts} we state their characterizations in terms of
drift conditions; in Section \ref{sect2} we prove the main tool
which is the exponential inequality for almost independent random
variables; finally in Section
\ref{sect3} previous results are combined to prove exponential
concentration inequalities for Markov chains.

\paragraph{Acknowledgement} We would like to thank Krzysztof {\L}atuszy{\'n}ski for interesting discussions and providing us with references to the statistical literature.

\section{The exponentially fast decaying tails of Markov Chain excursions}\label{sect1}

\subsection{Notation and preliminaries}  Since we are
interested in exponential inequalities let us first recall the
definition of exponential Orlicz norms,
$$
\|Y\|_{\psi_{\alpha},\mu}=\inf\Big\{c>0:\;\E_{\mu}\exp(\frac{|Y|^{\alpha}}{c^{\alpha}})\ls 2\Big\}.
$$
Note that the subscript above indicates the measure with respect to which the Orlicz norm is taken. Sometimes, when the underlying measure is clear from the context or when it is not relevant we will also write $\|\cdot\|_{\psi_\alpha}$.

Let us now briefly recall the regeneration method. We outline only
the main points needed for our further applications and refer to
\cite{Num1,MeynTwe} for an extensive exposition.

\smallskip

\noindent The split chain construction is based on introducing
variables $Y_k\in\{0,1\}$ that denote the level of the split
$m$-skeleton at time $km$.  Let $\bar{\bf
X}^m=\{\bar{X}^m_k:\;k\in \Z_{+}\}=\{(X_{km},Y_k):\;k\in Z_{+}\}$.
The new chain is defined by conditional probabilities
\begin{align*}
&\bar{\P}(Y_k=1,X_{km+1}\in dx_1,...,X_{(k+1)m-1}\in dx_{m-1},X_{(k+1)m}\in dy|\ccF^X_{km},\ccF^Y_{k-1},X_{km}=x)=\\
&=\bar{\P}(Y_k=1,X_{1}\in dx_1,...,X_{m-1}\in dx_{m-1},X_{m}\in dy|X_{0}=x)=\\
&=1_{C}(x)\frac{\delta
\nu(dy)}{\P^m(x,dy)}\P(x,dx_1)...\P(x_{m-1},dy)
\end{align*}
and
\begin{align*}
&\bar{\P}(Y_k=0,X_{km+1}\in dx_1,...,X_{(k+1)m-1}\in dx_{m-1},X_{(k+1)m}\in dy|\ccF^X_{km},\ccF^Y_{k-1},X_{km}=x)=\\
&=\bar{\P}(Y_k=0,X_{1}\in dx_1,...,X_{m-1}\in dx_{m-1},X_{m}\in dy|X_{0}=x)=\\
&=\Big(1- 1_{C}(x)\frac{\delta
\nu(dy)}{\P^m(x,dy)}\Big)\P(x,dx_1)...\P(x_{m-1},dy),
\end{align*}
where $\mathcal{F}_{km}^X = \sigma((X_i)_{i\le km})$ and
$\mathcal{F}^Y_{k-1} = \sigma((Y_i)_{i\le k-1})$. The above
condition means simply that we arrange the conditional
distribution of the intermediate parts of the chain so that they
fit to the split $m$-skeleton. Note that
$$
\bar{\P}(Y_k=1,X_{(k+1)m}\in dy|\ccF^X_{km},\ccF^Y_{k-1},X_{km}=x)=1_C(x)\delta\nu(dy),
$$
and consequently $\bar{\P}( Y_k=1|\ccF^X_{km},\ccF^Y_{k-1},
X_{km}=x)=\delta 1_C(x)$ and $\bar{\P}(X_{(k+1)m}\in
dy|\ccF^{X}_{km},\ccF^Y_{k}, Y_k=1)=\nu(dy)$. Therefore whenever
$X_{km}$ enters $C$, with probability $\delta$ we decide on
$Y_k=1$, and if so distribute $X_{(k+1)m}$ according to the
measure $\nu$.

One also checks easily that $X_i$ is a Markov chain with
transition function $P$.

For each initial measure $\mu$ we denote by $\mu^{\ast}$ the
measure on $\ccX\times\{0,1\}$ such that
$\mu^{\ast}(A\times\{0\})=\mu(A\cap C)(1-\delta)+\mu(A\cap C^{c})$
and $\mu^{\ast}(A\times \{1\})=\mu(A\cap C)\delta$. We continue
convention $\bar{\P}_{x^{\ast}}=\bar{\P}_{\delta_{x}^{\ast}}$, where $\delta_x$ stands for the Dirac mass at $x$. The
clear consequence of the construction is that
$$
\{X_i,Y_{j}:\;i\ls km,j\ls k \}\;\;\mbox{is independent of}\;\;\{X_i,Y_j:\;i\gs (k+1)m,j\gs k+1 \}
$$
under the condition that $Y_k=1$. Moreover $\{X_i,Y_j:\;i\gs (k+1)m,j\gs k+1 \}$  has the same distribution as
$\bar{\P}_{\nu^{\ast}}$ process $\{X_i,Y_j,\;i,j\gs 0\}$, where
$$
\bar{\P}_{\nu^{\ast}}(Y_0=1,X_0\in dx)=\delta 1_C(x)\nu(dx).
$$

Therefore we can treat $\theta=C\times\{1\}$ like an atom of the
chain $\bar{\bf X}^m$. Our approach to deviation inequalities will
be based on the decomposition of the sum $\sum^{n-1}_{i=0}f(X_i)$ into
almost independent excursions between consecutive return times to
$\theta$.

Let $\sigma=\sigma(0)=\min\{k\gs 0:\;Y_k=1\}$ and
$$
\sigma(i)=\min\{k>\sigma(i-1):\;Y_k=1\},\;\;k>0
$$
and in the same way $\tau=\tau(1)=\min\{k\gs 1:\;Y_k=1\}$ and
$$
\tau(i)=\min\{k>\tau(i-1):Y_k=1\},\;\; k>1.
$$
For each $i$ we define
$$
s_i(f)=\sum^{m\sigma(i+1)+m-1}_{j=m(\sigma(i)+1)}f(X_j)=\sum^{\sigma(i+1)}_{j=\sigma(i)+1}Z_j(f),
$$
where $Z_j(f)=\sum^{m-1}_{k=0}f(X_{jm+k})$. The main result on the excursions is the following (\cite[Theorem 17.3.1]{MeynTwe})
\bt
The two collections of random variables
$$
\{s_i(f):\;0\ls i\ls k-2\},\;\;\{s_i(f):\;i\gs k\}
$$
are independent for any $k\gs 2$. The distribution of $s_i(f)$ is for any $i$ equal to the $\bar{\P}_{\theta}$ distribution of
$\sum^{\tau m+(m-1)}_{k=m}f(X_k)$ which is equal to $\bar{\P}_{\nu^{\ast}}$ distribution of $\sum^{\sigma m+(m-1)}_{k=0}f(X_k)$.
Moreover the common mean of $s_i(f)$ may be expressed as $\bar{\E}s_i(f)=\delta^{-1}\pi(C)^{-1}m\int f d\pi$.
\et
Therefore if $\pi(f)=0$ then also $\bar{\E}s_i(f)=0$. We use the above result to decompose the path into three parts
\ba
\label{decomp} && \Big|\sum^{n-1}_{k=0}f(X_k)\Big|\le \Big|\sum^{\min(m\sigma+(m-1),n-1)}_{k=0}f(X_k)\Big|+\Big|\sum^{N}_{i=1}s_{i-1}(f)\Big|+\Big|1_{N >0}\sum^{m\sigma(N)+(m-1)}_{k=n}f(X_k)\Big|\nonumber \\
&&=:U_n(f)+V_n(f)+W_n(f),\ea and
\begin{align}\label{N_def}
N=\inf\{i\gs
0:\;m\sigma(i)+(m-1)\gs n-1\}.
\end{align} Note that $N$ is a stopping time
with respect to the filtration $\ccF_i =
\sigma(s_0,\ldots,s_{i-1},\sigma(0),\ldots,\sigma(i))$.

In order to prove exponential inequalities we need to provide
appropriate integrability conditions for all the summand above.
They will be expressed in terms of the following exponential
norms:
\begin{align}\label{definition_abc}
& \cca=\|\sum^{\sigma}_{i=0}|Z_i(f)|\|_{\psi_{\alpha},\bar{\P}_{x^{\ast}}}<\infty;\nonumber\\
& \ccb=\|\sum^{\sigma}_{i=0}|Z_i(f)|\|_{\psi_{\alpha},\bar{\P}_{\pi^{\ast}}}<\infty;\nonumber\\
& \ccc=\|s_i(f)\|_{\psi_{\alpha},\bar{\P}}<\infty.
\end{align}

The above quantities are rather troublesome to use (as expressed in terms of the split chain on the enlarged probability space), therefore we
prefer their counterparts expressed in terms of the original
chain, i.e. without referring to the auxiliary variables $Y_i$.
They are
\begin{align*}
& \ccA=\|\sum^{\tau_C-1}_{i=0}| Z_i(f)|\|_{\psi_{\alpha},\P_x};\;\; \ccB=\|\sum^{\tau_C-1}_{i=0} |Z_i(f)|\|_{\psi_{\alpha},\P_{\pi}};\\
& \ccC=\sup_{x\in C}\|\sum^{\tau_C-1}_{i=1}Z_i(f)\|_{\psi_{\alpha},\P_x};\;\;\ccD=\sup_{x\in C}\|Z_0(f)\|_{\psi_{\alpha},\P_x},
\end{align*}
where $\tau_C = \tau_C(1)= \inf\{k\ge 1\colon X_{km} \in C\}$. For
later use define also $\tau_C(i) = \inf\{k > \tau_C(i-1)\colon
X_{km} \in C\}$ for $i> 1$.

Note that for $m=1$, $\ccD = \sup_{x\in C} |f(x)|$.

 Let $r\gs 1$ be the unique solution of
\be\label{wrona0}
2^{\frac{1}{r}}\delta^{1-\frac{1}{r}}+2^{1+\frac{1}{r}}(1-\delta)^{1-\frac{1}{r}}=2
\ee
(recall that $\delta$ is the
number appearing in the minorization condition (\ref{small})).

In particular if $\delta=1$, then $r=1$. Moreover by the concavity of $x^{1-\frac{1}{r}}$  and monotonicity of the left hand side above in $r \in [1,\infty)$ we get
$r\ls \log(\frac{6}{2-\delta})/\log(\frac{2}{2-\delta})$.

The following proposition provides a comparison between
$\cca,\ccb,\ccc$ and $\ccA,\ccB,\ccC,\ccD$.

\begin{prop}\label{pros1}In the setting described above with $\alpha \in (0,1]$ the following inequalities hold:
\begin{align*}\label{pro1}
& \cca\ls  r^{\frac{1}{\alpha}}((\max\{\ccA,\ccC\})^{\alpha}+\ccD^{\alpha})^{\frac{1}{\alpha}};\\
& \ccb \ls r^{\frac{1}{\alpha}}((\max\{\ccB,\ccC\})^{\alpha}+\ccD^{\alpha})^{\frac{1}{\alpha}};\\
& \ccc \ls r^{\frac{1}{\alpha}}(\ccC^{\alpha}+\ccD^{\alpha})^{\frac{1}{\alpha}}.
\end{align*}
\end{prop}

\begin{proof}
Recall that distributions of $s_i(f)$ are the same for all $i$ and
equal to $\nu^{\ast}$ distribution of $\sum^{\sigma}_{k=0}Z_k(f)$. Let $a \in C$ be an arbitrary point. We will write $\bar{\E}_{a,i}$ ($i=0,1$) to denote the conditional expectation (on the enlarged probability space) given $X_0 = a, Y_ 0 = i$. In particular by the construction of the split process the distribution of $(X_m,X_{m+1},\ldots)$ under $\bar{\E}_{a,1}$ is independent of $a \in C$ (and equal to the $\bar{\E}_{\nu^\ast}$ distribution of $(X_0,X_1,\ldots)$). Thus by standard conditioning arguments, the following inequalities hold for any $c >0$
\begin{align*}
& \bar{\E}\exp(c^{-\alpha}|s_i(f)|^{\alpha})\ls \\
&\ls  \bar{\E}_{a,1}\exp(c^{-\alpha}|\sum^{\tau_C-1}_{j=1}Z_j(f)|^{\alpha})\\
&\phantom{aa}\times
\sum^{\infty}_{k=1}[(1-\delta)\sup_{x\in C}\bar{\E}_{x,0}\exp(c^{-\alpha}|\sum^{\tau_C-1}_{j=0}Z_j(f)|^{\alpha})]^{k-1}\sup_{x\in C}\delta\bar{\E}_{x,1}\exp(c^{-\alpha}|Z_0(f)|^\alpha).
\end{align*}
Therefore (for $\delta$ sufficiently small)
\begin{align*}
& \bar{\E}\exp(c^{-\alpha}|s_i(f)|^{\alpha})\ls \\
&\ls \frac{\delta\bar{\E}_{a,1}\exp(c^{-\alpha}|\sum^{\tau_C-1}_{j=1}Z_j(f)|^{\alpha})\sup_{x\in C}\bar{\E}_{x,1}\exp(c^{-\alpha}|Z_0(f)|^{\alpha})}{1-(1-\delta)\sup_{x\in C}\bar{\E}_{x,0}\exp(c^{-\alpha}|\sum^{\tau_C-1}_{j=0}Z_j(f)|^{\alpha})}.
\end{align*}
Set $c=r^{\frac{1}{\alpha}}(\ccC^{\alpha}+\ccD^{\alpha})^{\frac{1}{\alpha}}$.
Let $p^{-1}=\frac{\ccC^{\alpha}}{\ccC^{\alpha}+\ccD^{\alpha}}$
and $q^{-1}=\frac{\ccD^{\alpha}}{\ccC^{\alpha}+\ccD^{\alpha}}$. Recall that $r \ge 1$. By H\"{o}lder's inequality
\begin{align*}
&\sup_{x\in C}\bar{\E}_{x,0}\exp(c^{-\alpha}|\sum^{\tau_C-1}_{j=0}Z_j(f)|^{\alpha})\\
&\phantom{aaa}\ls  (\sup_{x\in C}\bar{\E}_{x,0}\exp(\ccD^{-\alpha}|Z_0(f)|^{\alpha}))^{\frac{1}{rq}}
(\sup_{x\in C}\bar{\E}_{x,0}\exp(\ccC^{-\alpha}|\sum^{\tau_C-1}_{j=1}Z_j(f)|^{\alpha}))^{\frac{1}{rp}},\\
&\sup_{x\in C}\bar{\E}_{x,1}\exp(c^{-\alpha}|Z_0(f)|^{\alpha})\ls (\sup_{x\in C}\bar{\E}_{x,1}\exp(\ccD^{-\alpha}|Z_0(f)|^{\alpha}))^{\frac{1}{qr}}
\end{align*}
and
$$
\bar{\E}_{a,1}\exp(c^{-\alpha}|\sum^{\tau_C-1}_{j=1}Z_j(f)|^{\alpha})\ls (\bar{\E}_{a,1}\exp(\ccC^{-\alpha}|\sum^{\tau_C-1}_{j=1}Z_j(f)|^{\alpha}))^{\frac{1}{pr}}.
$$
Define
$$
X=\sup_{x\in C}\bar{\E}_{x,0}\exp(\ccC^{-\alpha}|\sum^{\tau_C-1}_{j=1}Z_j(f)|^{\alpha}),\;\;Y=\bar{\E}_{a,1}\exp(\ccC^{-\alpha}|\sum^{\tau_C-1}_{j=1}Z_j(f)|^{\alpha})
$$
and
\begin{displaymath}
V=\sup_{x\in C}\bar{\E}_{x,0}\exp(\ccD^{-\alpha}|Z_0(f)|^{\alpha}),\;\;W=\sup_{x\in C}\bar{\E}_{x,1} \exp(\ccD^{-\alpha}|Z_0(f)|^{\alpha}).
\end{displaymath}
Using the above inequalities we get
$$
\bar{\E}\exp(c^{-\alpha}|s_i(f)|^{\alpha})\ls\frac{\delta Y^{\frac{1}{pr}}W^{\frac{1}{qr}}}{1-(1-\delta)X^{\frac{1}{pr}}V^{\frac{1}{qr}}}.
$$
By the definition of $\ccC, \ccD$ and the split chain construction
\be\label{wrona2}
(1-\delta)X+\delta Y= \sup_{x\in C} \E_x \exp(\ccC^{-\alpha}|\sum^{\tau_C-1}_{j=1}Z_j(f)|^{\alpha})\ls 2
\ee
and
\begin{align}\label{wrona1}
(1-\delta)V\ls \sup_{x\in C}\E_x \exp(\ccD^{-\alpha}|Z_0(f)|^{\alpha})\ls 2\;\;\mbox{and}\;\;\delta W\ls \sup_{x\in C}\E_x \exp(\ccD^{-\alpha}|Z_0(f)|^{\alpha})\ls 2.
\end{align}
Thus by the definition of $r$,
$$
\bar{\E}\exp(c^{-\alpha}|s_i(f)|^{\alpha})\ls \frac{\delta^{1-\frac{1}{r}}2^{\frac{1}{r}}}{1-(1-\delta)^{1-\frac{1}{r}}2^{\frac{1}{r}}} = 2,
$$
which proves the required bound on $\ccc$.

\noindent We will now estimate  $\cca$ and $\ccb$. In fact we will prove a more general statement. Consider any probability measure $\mu$ on $\mathcal{X}$ and denote $M_\mu = \|\sum_{i=0}^{\sigma_C-1} |Z_i(f)|\|_{\psi_\alpha,\bar{\P}_{\mu^\ast}}$, where $\sigma_C = \inf\{i\ge 0\colon X_{im} \in C\}$.
 Note that $M_\mu \le \|\sum_{i=0}^{\tau_C-1} |Z_i(f)|\|_{\psi_\alpha,\P_{\mu}}$. In particular $M_\pi \le \ccB$ and $M_{\delta_x} \le \ccA$. Thus to end the proof it is enough to show that
 \begin{align}\label{general_ineq_aux}
 \| \sum_{i=0}^\sigma |Z_i(f)|\|_{\psi_\alpha,\bar{\P}_{\mu^\ast}} \le r^{1/\alpha}((\max\{M_\mu,\ccC\})^\alpha + \ccD^\alpha)^{1/\alpha}.
 \end{align}

  Arguing as above we obtain that for any $c > 0$,
\begin{align*}
\bar{\E}_{\mu^{\ast}}\exp(c^{-\alpha}|s_0(f)|^{\alpha})\ls
\frac{\delta\E_\mu \exp(c^{-\alpha}(\sum_{i=0}^{\sigma_C-1}|Z_i(f)|)^\alpha)\sup_{y\in C}\bar{\E}_{y,1}\exp(c^{-\alpha}|Z_0(f)|^{\alpha})}{1-(1-\delta)\sup_{y\in C}\bar{\E}_{y,0}\exp(c^{-\alpha}(\sum^{\tau_C-1}_{j=0}|Z_j(f)|)^{\alpha})}
\end{align*}
Set $c=r^{\frac{1}{\alpha}}((\max\{M_\mu,\ccC\})^{\alpha}+\ccD^{\alpha})^{\frac{1}{\alpha}}$.
Let $p^{-1}=\frac{(\max\{M_\mu,\ccC\})^{\alpha}}{(\max\{M_\mu,\ccC\})^{\alpha}+\ccD^{\alpha}}$
and $q^{-1}=\frac{\ccD^{\alpha}}{(\max\{M_\mu,\ccC\})^{\alpha}+\ccD^{\alpha}}$.
Applying the notation of $X,V,W$ and H\"{o}lder's inequality we get
\begin{align*}
\bar{\E}_{\mu^{\ast}}\exp(c^{-\alpha}|s_0(f)|^{\alpha})\ls \frac{\delta (\E_\mu \exp(M_\mu^{-\alpha}(\sum_{i=0}^{\sigma_C-1}|Z_i(f)|)^\alpha))^{\frac{1}{pr}}W^{\frac{1}{qr}}}{1 - (1-\delta)X^{\frac{1}{pr}}V^{\frac{1}{qr}}}
\end{align*}
and thus by (\ref{wrona2}) and (\ref{wrona1}) we deduce that
$$
\bar{\E}_{\mu^{\ast}}\exp(c^{-\alpha}|s_0(f)|^{\alpha})\ls\frac{\delta^{1-\frac{1}{qr}}2^{\frac{1}{pr}}2^{\frac{1}{qr}}}{1-(1-\delta)^{1-\frac{1}{r}}2^{\frac{1}{r}}}
\ls\frac{\delta^{1-\frac{1}{r}}2^{\frac{1}{r}}}{1-(1-\delta)^{1-\frac{1}{r}}2^{\frac{1}{r}}} = 2
$$
by the definition of $r$. This ends the proof of (\ref{general_ineq_aux}).
\end{proof}

\section{Drift conditions\label{section_drifts}}
Our next goal is to provide conditions guaranteeing that
$\ccA,\ccB,\ccC,\ccD$ are finite, which via Proposition \ref{pros1} and
inequalities for independent summands will allow us to control the
quantities $U_n(f),V_n(f),W_n(f)$ in (\ref{decomp}).

One of the standard tools that has proved useful in the analysis
of integrability properties for the excursions of Markov chains is
drift conditions.

Below we consider two types of drift criteria. The first one is the multiplicative drift condition
introduced in \cite{MeynKos1,MeynKos2} to deal with pure
exponential integrability (i.e. with $\alpha = 1$). It is known \cite{MeynKos2} that in the case of $m=1$ this condition is equivalent
to exponential integrability (finiteness of $\psi_1$ norms) of $\sum_{i=0}^{\tau_C-1} f(X_i)$. Below we analyze the drift condition expressed for the $m$-skeleton and properly modified function, obtaining sufficient conditions for the finiteness of the parameters $\ccA,\ccB,\ccC,\ccD$.  We also show that the multiplicative drift condition is in a sense a minimal requirement for proving exponential integrability of the excursion, in particular obtaining good constants in the equivalence proved in \cite{MeynKos2}.

The multiplicative drift condition, although important from theoretical point of view, is of limited use in applications, as it is difficult to verify when compared to classical drift criteria used to obtain integrability of the regeneration time. Therefore we subsequently analyze what integrability properties of the excursion can be obtained with such simplified drift conditions. We show that by assuming integrability conditions on the regeneration time and a simple drift condition on the function $f$, one can still obtain $\psi_\alpha$ integrability, however with $\alpha < 1$. This still allows to obtain meaningful exponential inequalities, which for moderate values of $t$ agree with the classical Bernstein's bound.

\subsection{Multiplicative geometric drift condition\label{sect_milt_dift}}
Observe that
for any initial measure  $\nu$  and $\alpha\in(0,1]$,
\be\label{cond2} \|\sum^{\tau_C-1}_{k=0}
Z_k(f)\|_{\psi_{\alpha},\nu}\le
\|\sum^{\tau_C-1}_{k=0}|Z_k(f)|^{\alpha}\|^{\frac{1}{\alpha}}_{\psi_{1},\nu}.
\ee

In particular, in order to bound $\cca,\ccb,\ccc$ it suffices to
control the usual exponential $\psi_1$ norms of
$\sum^{\tau_C-1}_{k=0}|Z_k(f)|^{\alpha}$. Let thus denote
$g(x)=\log \E_x \exp(2c^{-1}|Z_0(f)|^{\alpha})$ and assume that for
some $c$ it verifies the multiplicative geometric drift condition
from \cite{MeynKos1,MeynKos2},  i.e. that there exists a function
$V:\ccX\ra \R_{+}$ and constants $b\gs 0,K>0$ such that
\be\label{cond3} \exp(-V(x))P^m(\exp(V))(x)\ls \exp(-g(x)+b
1_{C}(x)), \ee and $V(x)\ls K$ for $x\in C$.

The drawback of the drift condition (\ref{cond3}) in comparison
with the usual drift criteria is that in practice its direct
verification is difficult.
At the same time it turns out that it
is in fact equivalent to the existence of $c<\infty$ such that
\be\label{cond2a}
\sup_{x\in
C}\E_x\exp(\sum^{\tau_C-1}_{k=0}g(X^m_k)) =: d <\infty\;\textrm{and}\; \forall_{x\in \mathcal{X}}\,\|\sum_{k=0}^{\tau_C-1}g(X_k^m)\|_{\psi_1,\P_x} < \infty,
\ee
we recall that ${\bf X^m}=\{X_k^m:\;k\in\Z_{+}\}$ is the $m$-skeleton of ${\bf X}$.
Note that (\ref{cond2a}) implies
\be\label{cond2b}
\sup_{x\in C}\|\sum^{\tau_C-1}_{k=0}|Z_k(f)|^{\alpha}\|_{\psi_1,\P_x}<\infty \;\textrm{and}\; \forall_{x\in \mathcal{X}}\,\|\sum_{k=0}^{\tau_C-1}g(X_k^m)\|_{\psi_1,\P_x} < \infty
\ee
due to the Schwarz inequality
\begin{align}\label{GZ_comp}
&\E_x\exp(c^{-1}\sum^{\tau_C-1}_{k=0}|Z_k(f)|^{\alpha})\ls\E_x\exp(c^{-1}\sum^{\tau_C}_{k=0}|Z_k(f)|^{\alpha}) \\
&\ls [\E_x \exp( \sum^{\tau_C}_{k=0}g(X^m_k))]^{\frac{1}{2}}
[\E_x \exp( \sum^{\tau_C}_{k=0}(2c^{-1}|Z_k(f)|-g(X^m_k)))]^{\frac{1}{2}}\nonumber\\
&\le [\E_x\exp( \sum^{\tau_C}_{k=0}g(X^m_k))]^{\frac{1}{2}}\le
 [\E_x\exp( \sum^{\tau_C-1}_{k=0}g(X^m_k))\sup_{y\in C} \exp(g(y))]^{\frac{1}{2}}\nonumber\\
 &\le [\E_x\exp( \sum^{\tau_C-1}_{k=0}g(X^m_k)) \sup_{y\in C}\E_y \exp( \sum^{\tau_C-1}_{k=0}g(X^m_k))]^{1/2}\nonumber.
\end{align}

Moreover, if $m=1$ then (\ref{cond2a}) and (\ref{cond2b}) are equivalent (note that in this case $g(X_k^1) =2c^{-1}| f(X_k^1)|^\alpha$ and so we do not need the additional argument above).

We will now prove that the drift condition for $g$ is of the same power as  (\ref{cond2a}). More precisely, we have the following

\bt\label{thmcond1} Conditions (\ref{cond3}) and (\ref{cond2a})
are equivalent in the sense that
\begin{enumerate}
\item Whenever (\ref{cond3}) holds then so does (\ref{cond2a}).
Moreover for every $x \in \mathcal{X}$,
\begin{displaymath}
\E_x \exp(\sum^{\tau_C-1}_{k=0}g(X_k^m))\le
\exp(b1_C(x) + V(x))
\end{displaymath}
and consequently
$$
\sup_{x\in
C}\|\sum^{\tau_C-1}_{k=0}|Z_k(f)|^{\alpha}\|_{\psi_{1},\P_x}\ls
\max\{1,\frac{b+K}{\log 2}\}c.
$$
\item Whenever (\ref{cond2a}) holds, then
the function $g$ satisfies (\ref{cond3}) with $V(x)=\log(G_C(x,g))$,
where
$$
G_C(x,g)=\E_x\exp(\sum^{\sigma_C}_{k=0} g(X^m_k)),
$$
and $b=2\log d$, $K=\log d$.
\end{enumerate}
\et
\begin{proof}
Suppose that (\ref{cond3}) holds. Let
$\ccF_k=\sigma(X^m_0,X^m_1,...,X^m_k)$, $k=0,1,2,...$,  and define
the exponential martingale \be\label{marti}
M_k=\frac{\exp(V(X^m_k))}{\E(\exp(V(X^m_{k})|\ccF_{k-1}))}...\frac{\exp(V(X^m_1))}{\E(\exp(V(X^m_1))|\ccF_0)}\exp(V(X^m_0)).
\ee Therefore for the stopping time $\tau_C \wedge n$ we have
\be\label{nier} \E_{x} M_{\tau_C\wedge
n}=\exp(V(x)),\;\;\mbox{for}\;x\in \mathcal{X}. \ee
Due to (\ref{cond3}) we
obtain that
$$
\frac{\exp(V(X^m_{i-1}))}{\E_x(\exp(V(X^m_i))|\ccF_{i-1})}=\frac{\exp(V(X^m_{i-1}))}{P^m(\exp(V))(X^m_{i-1})}\gs
 \exp(g(X^m_{i-1}) - b1_C(X_{i-1}^m))
$$
and hence
$$
M_k\exp(-V(X^m_k))=\prod^{k}_{i=1}\frac{\exp(V(X^m_{i-1}))}{\E(\exp(V(X^m_i))|\ccF_{i-1})}\gs \exp(\sum^{k}_{i=1}g(X^m_{i-1})-b1_C(X^m_{i-1})).
$$
Consequently by (\ref{marti}) and (\ref{nier}) for every $x \in
\mathcal{X}$,
$$
\E_x \exp( \sum^{(\tau_C\wedge n)-1}_{k=0}g(X^m_k))\ls
\exp(b1_C(x)+V(x)),
$$
therefore by letting $n \to\infty$ and using Fatou's lemma we get
$$
\E_x \exp(\sum^{\tau_C-1}_{k=0}g(X^m_k))\ls
\exp(b1_C(x)+V(x)).
$$
Now, by (\ref{GZ_comp})
\begin{displaymath}
\E_x\exp(c^{-1}\sum_{k=1}^{\tau_C-1}|Z_k(f)|^\alpha) \le \exp((b1_C(x) + V(x)+b+K)/2),
\end{displaymath}
which via H\"older's inequality implies the second inequality of point 1.

To prove the second assertion observe that
\begin{align*}
&(P^m G_C)(x)=\E_x \exp(\sum^{\tau_C}_{k=1}g(X_k^m))=\\
&= \exp(-g(x))\E_x\exp(\sum^{\sigma_C}_{k=0}g(X_k^m))
\exp(1_C(x)\log(\E_x\exp(\sum^{\tau_C}_{k=1}g(X_k^m))))=\\
&=\exp(-g(x)+1_C(x)\log(\E_x\exp(\sum^{\tau_C}_{k=1}g(X_k^m)))G_C(x,g),
\end{align*}
i.e.
\begin{displaymath}
\exp(-V(x))(P^m \exp(V))(x) = \exp\Big(-g(x) + 1_C(x)\log(\E_x\exp(\sum^{\tau_C}_{k=1}g(X_k^m))) \Big).
\end{displaymath}
From the definition of $d$ in (\ref{cond2a}) we conclude that $g(x)\ls \log d$ for
$x\in C$ and
$$
\sup_{x\in C}\E_x \exp(\sum^{\tau_C}_{k=1}g(X_k))\ls d^2.
$$
Therefore
$$
\sup_{x\in
C}\log(\E_x\exp(\sum^{\tau_C}_{k=1}g(X_k)))\ls 2\log d.
$$
and
$$
\sup_{x\in C}V(x) = \log(G_C(x,g))=g(x)\ls \log d,\;\;\mbox{for}\;x\in C.
$$
This shows that (\ref{cond2a}) holds with $b=2\log d$ and $K=\log
d$.
\end{proof}

\begin{coro} If (\ref{cond3}) is satisfied then
\begin{align*}
\mathcal{A} & \le \Big(\frac{\max\{2b+V(x)+ K,2\log 2\}}{2\log 2}c\Big)^{\frac{1}{\alpha}},\\
\mathcal{B} & \le \Big(\frac{\max\{2b+2\log\pi(\exp(V)/2)+K,2\log 2\}}{2\log 2}c\Big)^{\frac{1}{\alpha}},\\
\mathcal{C},\ccD & \le \Big(\frac{\max\{b+K,\log 2\}}{\log
2}c\Big)^{\frac{1}{\alpha}}.
\end{align*}
\end{coro}

\begin{proof}
It is enough to combine the estimate
$\ccC^\alpha, \ccD^\alpha \le \sup_{x\in
C}\|\sum^{\tau_C-1}_{k=0}|Z_k(f)|^{\alpha}\|_{\psi_{1},\P_x}$, the first part of Theorem \ref{thmcond1}
(integrated with respect to $\pi$ in the case of the quantity
$\mathcal{B}$), observation (\ref{GZ_comp}) and the H\"older inequality.
\end{proof}

\subsection{'Regular' drift condition}\label{secRegularDrift}
\noindent As we have already mentioned, the multiplicative drift
condition is difficult to check. Therefore we would like to
replace it with a simpler criterion such as the usual drift
condition (see \cite{MeynTwe}, \cite{Num1}). The price to pay is
strengthening the requirements on $\tau_C$. Namely we
assume that
$$
\sup_{x\in C}\|\tau_C\|_{\psi_{\beta},\P_x}<\infty,
$$
where $\beta> \alpha$. Note that if $\beta=1$ then we are in the
setting of geometric ergodicity. We would like to point out that one can verify such integrability conditions on $\tau_C$ by using classical drift criteria (see e.g. \cite{MeynTwe}) for $\beta = 1$ or its modified versions presented e.g. in \cite{DGM} for $\beta<1$. In the latter reference also concentration inequalities for bounded $f$ were presented. Our aim now is to provide drift conditions on $f$, simpler than those discussed in the previous section, which would complement the $\psi_\beta$ integrability of $\tau_C$ and yield strong exponential inequalities in the unbounded case.

Let
$h(x)=\log\E_x\exp(c^{-\gamma}|Z_0(f)|^{\gamma})$ ,
$\gamma=\frac{\alpha\beta}{\beta-\alpha}$ (note that for $m=1$ we have $h(x) = c^{-\gamma}|f(x)|^\gamma$). The drift condition we will consider is
of the form: suppose there exists $V:\ccX\ra \R_+, b >0$ such that
\be\label{war} (P^m V)(x)-V(x)\ls
-\exp(h(x))+b1_C(x), \ee and $V(x)\ls K$ for $x\in C$. The
usual martingale argument (see \cite{MeynTwe}) shows that for all
$x \in \mathcal{X}$,

\be\label{war1}
\E_x\sum^{\tau_C-1}_{k=0}\exp(\frac{|Z_k(f)|^{\gamma}}{c^{\gamma}})\ls
b+V(x). \ee

\bp For $\alpha \in (0,1]$ and $\beta > \alpha$, if (\ref{war}) is satisfied, then
$$
\sup_{x\in C}\|\sum^{\tau_C-1}_{k=0}|Z_k(f)|\|_{\psi_{\alpha},\P_x}\ls c_1c_2,
$$
where $c_1=\sup_{x\in C}\|\tau_C\|_{\psi_{\beta},\P_x}$, $c_2=c(\max(\frac{\log(b+K)}{\log 2},1))^{1/\gamma}$.
\ep
\begin{proof}
For $u,v>0$  the Young inequality holds, i.e.
$$
u^{\alpha}v^{\alpha}\ls  \frac{\alpha}{\beta} u^{\beta}+(1-\frac{\alpha}{\beta}) v^{\frac{\alpha\beta}{\beta-\alpha}},
$$
Consequently
\begin{align*}
&  \E_{x}\exp(c_1^{-\alpha}c_2^{-\alpha}|\sum^{\tau_C-1}_{k=0}Z_k(f)|^{\alpha})
 \ls
\E_{x}\exp((c_1^{-1}\tau_C)^{\alpha}(\tau^{-1}_C\sum^{\tau_C-1}_{k=0}c_2^{-1}|Z_k(f)|)^{\alpha})\\
& \ls \E_{x} \exp(\frac{\alpha}{\beta}\frac{\tau^{\beta}_C}{c_1^{\beta}})
\exp((1-\frac{\alpha}{\beta})(\tau^{-1}_C\sum^{\tau_C-1}_{k=0}c_2^{-1}|Z_k(f)|)^{\frac{\alpha\beta}{\beta-\alpha}}).
\end{align*}
Therefore by the H\"older inequality
\begin{align*}
&\E_{x}\exp(c_1^{-\alpha}c_2^{-\alpha}|\sum^{\tau_C-1}_{k=0}Z_k(f))|^{\alpha})\ls \\
&\ls (\E_x(\exp(\frac{\tau_C^{\beta}}{c_1^{\beta}})))^{\frac{\alpha}{\beta}}
(\E_x (\exp(\tau^{-1}_C\sum^{\tau_C-1}_{k=0}c_2^{-1}|Z_k(f)|))^{\frac{\alpha\beta}{\beta-\alpha}}))^{1-\frac{\alpha}{\beta}}.
\end{align*}
Note that if $c_1=\sup_{x\in C}\|\tau_C\|_{\psi_{\beta},\P_x}$, then
$$
\E_x\exp(\frac{\tau_C^{\beta}}{c_1^{\beta}})\ls 2,\;\;\mbox{for}\;x\in C.
$$
Now we will estimate the second term. We have
\begin{align*}
\E_x (\exp(\tau^{-1}_C\sum^{\tau_C-1}_{k=0}c_2^{-1}|Z_k(f)|)^{\gamma})\ls
\E_x \exp(\max_{0\ls k\ls \tau_C-1}\frac{|Z_k(f)|^{\gamma}}{c_2^{\gamma}})
\ls \E_x (\sum^{\tau_C-1}_{k=0}\exp(\frac{|Z_k(f)|^{\gamma}}{c^{\gamma}}))^{\frac{c^{\gamma}}{c_2^{\gamma}}}.
\end{align*}
Since $c_2=c(\max(\frac{\log(b+K)}{\log 2},1))^{1/\gamma}$, (\ref{war1}) implies that
\begin{align*}
 \E_x (\sum^{\tau_C-1}_{k=0}\exp(\frac{|Z_k(f)|^{\gamma}}{c^{\gamma}}))^{\frac{c^{\gamma}}{c_2^{\gamma}}}
\ls  (\E_x \sum^{\tau_C-1}_{k=0}\exp(\frac{|Z_k(f)|^{\gamma}}{c^{\gamma}}))^{\frac{c^{\gamma}}{c_2^{\gamma}}}\ls 2
\end{align*}
and so
$$
\sup_{x\in C}\E_{x}\exp(c_1^{-\alpha}c_2^{-\alpha}|\sum^{\tau_C-1}_{k=0}Z_k(f)|^{\alpha})\ls 2,
$$
which completes the proof.
\end{proof}
Basically the same idea can be used to bound $\ccA$ and $\ccB$. We
summarize the result in the following \bt\label{thmcond2} Whenever
the drift condition (\ref{war}) is satisfied the following inequalities
hold:
\begin{enumerate}
\item Let $c_1=\sup_{x\in C}\|\tau_C\|_{\psi_{\beta},\P_x}$, $c_2=c(\max(\frac{b+K}{\log 2},1))^{1/\gamma}$, then
$$
\ccC,\ccD \le\sup_{x\in C}\|\sum^{\tau_C-1}_{k=0}|Z_k(f)|\|_{\psi_{\alpha},\P_x}\ls c_1c_2,
$$
\item Let  $c_3=\|\tau_C\|_{\psi_{\beta},\P_{x}}$, $c_4=c(\max(\frac{\log(V(x)+b)}{\log 2},1))^{1/\gamma}$, then
$$
\ccA=\|\sum^{\tau_C-1}_{k=0}|Z_k(f)|\|_{\psi_{\alpha},\P_{x}}\ls c_3 c_4
$$
\item Let $c_5=\|\tau_C\|_{\psi_{\beta},\P_\pi}$, $c_6=(\max(\frac{\log(\pi V+b)}{\log 2},1))^{1/\gamma}$, then
$$
\ccB=\|\sum^{\tau_C-1}_{k=0}|Z_k(f)|\|_{\psi_{\alpha},\P_{\pi}}\ls c_5 c_6,
$$
\end{enumerate}
\et

If $\beta=1$ then $\gamma=\alpha/(1-\alpha)$ and the requirement
on $\tau_C$ is equivalent to geometric ergodicity. Therefore in
the simplest case of $m=1$ to obtain  meaningful bounds on the $\ccA,\ccB,\ccC,\ccD$ it is enough e.g. to assume the classical drift condition \eqref{eq:geometric_ergodicity_drift_new}
and a pointwise bound on the function generating the additive functional.

\begin{prop}\label{prop:geometric_erg_new}
Assume that \eqref{small} is satisfied with $m=1$ and that for some function $V\colon \mathcal{X} \to [1,\infty)$ and $\lambda \in (0,1)$, $b \in \R_+$ the drift condition \eqref{eq:geometric_ergodicity_drift_new} holds.
Set $K = \sup_{x\in C} V(x)$.
Assume moreover that $g\colon \mathcal{X} \to \R$ satisfies
$$
|g(x)|\ls \kappa\Big(\log  V(x)\Big)^{\frac{1-\alpha}{\alpha}}
$$
for some $\kappa > 0$ and set $f = g - \pi g$.
Then
\begin{align*}
\ccA\le& \ccA_{drift} := \kappa \frac{1}{\log \frac{1}{1-\lambda}} \max\Big(\frac{\log(V(x)1_{C^c}(x) + (b(1-\lambda)^{-1} + K)1_C(x))}{\log 2},1\Big) \times\\
&\phantom{aaaaaaa}\times \Big[\Big(\max(\frac{\log(V(x)\lambda^{-1}+b\lambda^{-1})}{\log 2},1)\Big)^{1-\alpha} + \frac{1}{(\log 2)^{1-\alpha}}(\pi|g|/\kappa)^\alpha\Big]^{1/\alpha},\\
\ccB\le& \ccB_{drift} := \kappa\frac{1}{\log \frac{1}{1-\lambda}} \max\Big(\frac{\log(\pi V + (b(1-\lambda)^{-1}+K)\pi(C))}{\log 2},1\Big)\times\\
&\phantom{aaaaaaaa}\times\Big[\Big(\max(\frac{\log(\pi V\lambda^{-1}+b\lambda^{-1})}{\log 2},1)\Big)^{1-\alpha} + \frac{1}{(\log 2)^{1-\alpha}}(\pi|g|/\kappa)^\alpha\Big]^{1/\alpha},\\
\ccC,\ccD  \le& \ccC_{drift} := \kappa\frac{1}{\log \frac{1}{1-\lambda}}\max\Big(\frac{\log(b(1-\lambda)^{-1}+K)}{\log 2},1\Big)\times\\
&\phantom{aaaaaaaa}\times \Big[\Big(\max(\frac{\log(b\lambda^{-1}+K\lambda^{-1})}{\log 2},1)\Big)^{1-\alpha} + \frac{1}{(\log 2)^{1-\alpha}}(\pi|g|/\kappa)^\alpha\Big]^{1/\alpha}.
\end{align*}
Moreover
\begin{displaymath}
\pi|g| \le \left\{
\begin{array}{l}
\kappa \Big(\log(b\pi(C)\lambda^{-1})\Big)^{\frac{1-\alpha}{\alpha}}\quad{\rm for}\;\alpha \ge 1/2,\\
\kappa \Big(\log(b\pi(C)\lambda^{-1}) + \frac{1-2\alpha}{\alpha}\Big)^{\frac{1-\alpha}{\alpha}} - \kappa\Big(\frac{1-2\alpha}{\alpha}\Big)^{\frac{1-\alpha}{\alpha}}\quad{\rm for}\; \alpha < 1/2
\end{array}
\right.
\end{displaymath}
and
\begin{align}\label{eq:piV_est_new}
\pi V \le \frac{b\pi(C)}{\lambda}.
\end{align}
\end{prop}

To prove the above proposition we will need the following

\begin{lema}\label{le:auxiliary_regeneration_new}
Assume that \eqref{small} is satisfied with $m=1$ and that for some function $V\colon \mathcal{X} \to [1,\infty)$ and $\lambda \in (0,1)$, $b \in \R_+$ the drift condition \eqref{eq:geometric_ergodicity_drift_new} holds.
Set $K = \sup_{x\in C} V(x)$. Then
\begin{align*}
\|\tau_C\|_{\psi_1,\P_x} &\le \max\Big(\frac{\log(V(x)1_{C^c}(x) + (b(1-\lambda)^{-1}+K)1_C(x))}{\log 2},1\Big)\frac{1}{\log \frac{1}{1-\lambda}},\nonumber\\
\|\tau_C\|_{\psi_1,\P_\pi} &\le \max\Big(\frac{\log(\pi V + (b(1-\lambda)^{-1}+K)\pi(C))}{\log 2},1\Big)\frac{1}{\log \frac{1}{1-\lambda}},\\
\sup_{x\in C} \|\tau_C\|_{\psi_1,\P_x} &\le \max\Big(\frac{\log(b(1-\lambda)^{-1}+K)}{\log 2},1\Big)\frac{1}{\log \frac{1}{1-\lambda}}.\nonumber
\end{align*}
\end{lema}

\begin{proof}
From Proposition 4.1. (ii) in \cite{Bax} we obtain
\begin{align*}
\E_x (1-\lambda)^{-\tau_C} \le \left\{
\begin{array}{ccc}
V(x)&\;\textrm{if}\;&x \notin C,\\
b(1-\lambda)^{-1}+K&\;\textrm{if}\;&x\in C.
\end{array}
\right.
\end{align*}
(note that due to different conventions regarding the drift condition `$\lambda$ in \cite{Bax}' is `our  $(1-\lambda)$', moreover `$K$ in \cite{Bax}' is trivially bounded by `our $(1-\lambda)K + b$').
This implies the lemma by integration and H\"older's inequality.
\end{proof}

We will also need the following well known lemma. Since we haven't been able to find its proof in the literature, we provide it for completeness.

\begin{lema}\label{le:Orlicz_new}
For any $\alpha \in (0,1]$ and any random variables $X,Y$ we have
\begin{displaymath}
\|X+Y\|_{\psi_\alpha} \le (\|X\|_{\psi_\alpha}^\alpha + \|Y\|_{\psi_\alpha}^\alpha)^{1/\alpha}.
\end{displaymath}
Moreover
\begin{displaymath}
\|X\|_{\psi_\alpha} \le \Big(\frac{1}{\log 2}\Big)^{\frac{1-\alpha}{\alpha}} \|X\|_{\psi_1}.
\end{displaymath}
\end{lema}
\begin{proof}
For $\alpha=1$, the first inequality reduces to the well known triangle inequality for $\|\cdot\|_{\psi_1}$. The general case follows easily from this special one and the identity $\|X\|_{\psi_\alpha} = \||X|^\alpha\|_{\psi_1}^{1/\alpha}$.

For the second inequality, note that by Young's inequality we have for $p > 1$ we have $|xy| \le p^{-1}|x|^p + q^{-1}|y|^q$, where $q$ is the conjugate of $p$, and so
by convexity, for any random variables $X,Y$,
\begin{displaymath}
\E \exp\Big( \frac{|XY|}{\|X\|_{\psi_p}\|Y\|_{\psi_q}}\Big) \le p^{-1}\E\exp\Big(\frac{|X|}{\|X\|_{\psi_p}}\Big) + q^{-1}\E\exp\Big(\frac{|Y|}{\|Y\|_{\psi_q}}\Big) \le 2,
\end{displaymath}
which implies that $\|XY\|_{\psi_1} \le \|X\|_{\psi_p}\|Y\|_{\psi_q}$. Now setting $Y = 1$ we get
\begin{displaymath}
\|X\|_{\psi_1} \le \|X||_{\psi_p} \Big(\frac{1}{\log 2}\Big)^{1/q}.
\end{displaymath}
\end{proof}

\begin{proof}[Proof of Propositon \ref{prop:geometric_erg_new}]
By Theorem 14.3.7. in \cite{MeynTwe} we have $\pi V < \infty$. By the drift condition \eqref{eq:geometric_ergodicity_drift_new} we get
$\pi V = \pi PV \le (1-\lambda)\pi V + b\pi(C)$, which implies \eqref{eq:piV_est_new}.

Set now $\tilde{V} = V\lambda^{-1}$, $\tilde{b} = b\lambda^{-1}$. Then \eqref{eq:geometric_ergodicity_drift_new} is satisfied also with $\tilde{V}, \tilde{b}$ instead of $V,b$ resp. Using this together with the assumption on $f$ one easily shows that for $\beta = 1$  the condition \eqref{war} with $\tilde{V},\tilde{b}$ instead of $V,b$ resp.  and $c=\kappa$ is also satisfied.

By the second part of Lemma \ref{le:Orlicz_new} we have $\|\tau_C \pi g\|_{\psi_\alpha} \le (\log 2)^{1-1/\alpha} \|\tau_C\|_{\psi_1} \pi|g|$. Thus
by the first part of the lemma
\begin{displaymath}
\|\sum_{k=0}^{\tau_C-1} |Z_k(f)|\|_{\psi_\alpha} \le \Big(\|\sum_{k=0}^{\tau_C-1} |Z_k(g)|\|_{\psi_\alpha}^\alpha + \frac{1}{(\log 2)^{1-\alpha}}\|\tau_C\|_{\psi_1}^\alpha (\pi|g|)^\alpha\Big)^{1/\alpha}.
\end{displaymath}
To prove the estimates on the parameters $\ccA,\ccB,\ccC,\ccD$ it is thus enough to use Theorem \ref{thmcond2} (with $\tilde{V},\tilde{b}$ and $\beta=1$) combined with the above bound and Lemma \ref{le:auxiliary_regeneration_new} (this lemma is applied with $V,b$ as it does not involve the function $g$).

It remains to prove the bound on $\pi|g|$.  For $\alpha \ge 1/2$ the function
$x \mapsto (\log x)^{(1-\alpha)/\alpha}$ is concave on $[1,\infty)$ (recall that by assumption $V \ge 1$), so the bound of Proposition \ref{prop:geometric_erg_new} follows from the assumption on $g$, Jensen's inequality and \eqref{eq:piV_est_new}. For $\alpha < 1/2$ the function $x \mapsto ((1-2\alpha)/\alpha + \log x)^{(1-\alpha)/\alpha}$ is concave. Thus we can write (using the inequality $(x+y)^s \ge x^s + y^s$ for $s \ge 1$ and $x,y \ge 0$),
\begin{displaymath}
\pi (\log V )^{(1-\alpha)/\alpha} \le \pi \Big(\log (V) + \frac{1-2\alpha}{\alpha}\Big)^{(1-\alpha)/\alpha} - \Big(\frac{1-\alpha}{\alpha}\Big)^{(1-\alpha)/\alpha}
\end{displaymath}
and apply Jensen's inequality to the first summand on the right hand side. The proof is again concluded by \eqref{eq:piV_est_new}.
\end{proof}

\paragraph{Remark} We would like to stress that drift conditions of the form \eqref{eq:geometric_ergodicity_drift_new} (with explicitly stated constants or constants which can be obtained from the proofs) have been verified for many models of practical interest, see e.g \cite{MengersenTweedie, RobertsTweedie, JarnerHansen,JohnsonJones,JonesHob}. There is also some literature concerning models for which Theorem \ref{thmcond2} can be applied with $\beta < 1$ (see \cite{DFMS,DGM}). In such models a modified drift condition holds, which implies a bound on $\|\tau\|_{\psi_\beta}$, which can be easily combined with Theorem \ref{thmcond2} to give a counterpart of Proposition \ref{prop:geometric_erg_new}. Since such a bound would depend on the form of the drift condition we restrict our attention to the most classical case given by \eqref{eq:geometric_ergodicity_drift_new}.

\paragraph{Examples}

We will now provide two examples related to potential applications of our results. They will both consider Markov chains used in the Metropolis-Hastings algorithm (on $\N$ and on $\R$). Let us first recall briefly the basic ideas behind this algorithm. Assume that $\pi$ is a density of a probability measure on $\mathcal{X}$, which is only known up to a factor, i.e. only the rations $\pi(x)/\pi(y)$ are known. Assume also that $Q(\cdot,\cdot)$ is a transition function of a Markov chain on $\mathcal{X}$ and that $Q(x,\cdot)$ has density $q(x,y)$. Then the Metropolis-Hastings algorithm of approximating $\pi$ relies on generating a Markov chain $X_n$, such that whenever $X_n = x$ then one first generates a new point $y$ according to $Q(x,\cdot)$ and sets $X_{n+1} = y$ or $X_{n+1} = x$ with probability respectively $\alpha(x,y)$ and $1-\alpha(x,y)$, where
\begin{displaymath}
\alpha(x,y) = \left\{\begin{array}{ccc}
\min\Big(\frac{\pi(y)q(y,x)}{\pi(x)q(x,y)},1\Big)&\;\textrm{if}\; &\pi(x)q(x,y) > 0\\
1&\;\textrm{if}\;&\pi(x)q(x,y) =0.
\end{array}
\right.
\end{displaymath}
One then easily checks that $X_n$ is a Markov chain admitting $\pi$ as an invariant measure.

\paragraph{Example 1. A sub-geometric measure on $\N$} Consider a probability measure $\pi$ on $\N$ (which we will identify with the sequence of weights $(\pi_i)_{i\in \N}$). Assume that for some $\rho < 1$ we have $0<\pi(i+1) \le \rho\pi(i)$. Let us consider a Metropolis-Hastings algorithm, starting from $0$ and such that
$q(0,1) = q(0,0) = 1/2$ and for $i > 0$, $q(i,i-1) = q(i,i+1) = 1/2$. Thus $C=\{0\}$ is an atom and \eqref{small} is satisfied with $\delta = 1$.

Consider a drift function given by $V(n) = A^{n+1}$ with $A> 1$ to be fixed later. We have for $i > 0$,
\begin{displaymath}
P V(i) \le \frac{1}{2}A^{i+1} + \frac{\rho}{2}A^{i+2} +\frac{1-\rho}{2}A^{i+1} \le (1-\lambda)V(i)
\end{displaymath}
for any $A \in (1, \rho^{-1})$ and
\begin{displaymath}
\lambda = \lambda_A = 1 - \frac{1}{2A} - \frac{\rho A}{2} - \frac{1-\rho}{2} > 0.
\end{displaymath}

Moreover, we have $P V(0) = (1-\rho/2)A + \rho A^2/2$. Thus the the drift condition \eqref{eq:geometric_ergodicity_drift_new} is satisfied with $\lambda$ defined above and $b = (A-1)/2$.
Note also that $K = V(0) = A$.
Consider now a function $g \colon \N \to \R$ such that for some $s > 0$ and $\kappa$,
\begin{displaymath}
|g(n)| \le \kappa (1+n)^s.
\end{displaymath}
Since  for $\alpha = 1/(s+1)$ we have $(1-\alpha)/\alpha = s$ and $\log V(n) = (1 + n)\log A$, we get
\begin{displaymath}
|g| \le \frac{\kappa}{(\log A)^s} \Big(\log V(n)\Big)^{\frac{1-\alpha}{\alpha}}.
\end{displaymath}
Using Proposition \ref{prop:geometric_erg_new} we obtain for $\tilde{\kappa} = \kappa/(\log A)^s$ and $f = g - \pi g$,
\begin{align*}
&\ccA, \ccC,\ccD  \le \ccA_{drift} = \ccC_{drift} = \tilde{\kappa}\frac{1}{\log \frac{1}{1-\lambda}}\max\Big(\frac{\log(2^{-1}(A-1)(1-\lambda)^{-1} + A)}{\log 2},1\Big) \times\\
&\times \Big[\Big(\frac{\log(A\lambda^{-1}+2^{-1}(A-1)\lambda^{-1})}{\log 2}\Big)^{\frac{s}{s+1}} + \frac{1}{(\log 2)^{s/(s+1)}} (\pi|g|/\tilde{\kappa})^{1/(s+1)}\Big]^{s+1},\\
&\ccB\le \ccB_{drift} = \tilde{\kappa}\frac{1}{\log \frac{1}{1-\lambda}}\max\Big(\frac{\log(\pi V + (2^{-1}(A-1)(1-\lambda)^{-1}+A)\pi_0)}{\log 2},1\Big)\times\\
&\times\Big[\Big(\frac{\log(\pi V\lambda^{-1}+2^{-1}(A-1)\lambda^{-1})}{\log 2}\Big)^{\frac{s}{s+1}}+\frac{1}{(\log 2)^{s/(s+1)}} (\pi|g|/\tilde{\kappa})^{1/(s+1)}\Big]^{s+1},\\
\end{align*}

Thus, using the fact that for $\delta = 1$, the $r$ defined by \eqref{wrona0} equals to 1, we get
\begin{align}\label{eq:example_geo}
\cca,\ccc \le 2^{s+1}\ccC_{drift},\ccb \le \Big((\max\{\ccA_{drift},\ccB_{drift}\})^{1/(s+1)} + \ccB_{drift}^{1/(s+1)}\Big)^{s+1}.
\end{align}

The quantities $\pi V$ and $\pi|g|$ can be estimated by Proposition \ref{prop:geometric_erg_new}, however in this case they can be also controlled just in terms of $\rho$ (by comparison with the geometric distribution with probability of success equal to $1-\rho$). We skip the standard details.

\paragraph{Example 2. Densities on $\R$ which are log-concave in tails} Let us now consider the Metropolis-Hastings algorithm of estimating an integral with respect to a probability measure on $\R$, which is log-concave in tails (see below for the definition). For such a measure an explicit drift condition \eqref{eq:geometric_ergodicity_drift_new} has been established in \cite{MengersenTweedie}. We remark that a generalization to measures on $\R^d$, satisfying some additional technical conditions (which can be easily verified for spherically symmetric measures) has been obtained in \cite{RobertsTweedie}. For clarity of the exposition we will stick to the one-dimensional case.

Here we will consider a special case of estimating a symmetric positive density $\pi$ on $\R$, which is log-concave in tails, i.e. such that there exists $\rho > 0$ and $x_1$ such that for all $y> x > x_1$
\begin{align}\label{eq:lc_in_tails_new}
\log \pi(x) - \log \pi(y) \ge \rho (y-x).
\end{align}

Moreover we will assume that $q(x,y) = q(x-y)$ for some symmetric positive density $q$ on $\R$. In this setting the arguments from \cite{MengersenTweedie} can be used to obtain an explicit form of drift function $V$ and constants $\lambda, b$ in \eqref{eq:geometric_ergodicity_drift_new}. We remark that the symmetry assumption on $\pi$ can be relaxed by imposing further restrictions on $q$. Since our purpose here is just to provide an illustration to our Proposition \ref{prop:geometric_erg_new} (and the concentration inequalities we will derive later), we will not consider the general case. Also, since this is just a question of scaling the measure, to simplify slightly the formulas we will assume that $\rho = 2$.

By  Lemmas 1.1. and 1.2. in \cite{MengersenTweedie} the chain is $\pi$-irreducible and strongly aperiodic and by the calculations in the proof of this lemma any compact set $C$ is a small set (i.e. it satisfies \eqref{small}), with $m=1$ and
\begin{align}\label{eq:form_of_delta_new}
\delta =  \pi(C)\frac{\inf_{x,y \in C} q(y-x)}{\sup_{x\in C} \pi(x)}
\end{align}
(with the measure $\nu = \pi(\cdot|C)$).
We will now set $C = [-x^\ast,x^\ast]$ for some $x^\ast > x_1$.
By arguments used in the proof of Theorem 3.2. in \cite{MengersenTweedie} one can easily obtain that \eqref{eq:geometric_ergodicity_drift_new} is satisfied with
\begin{align}\label{eq:MHdrift_new}
V(x) &= \exp(|x|+1),\nonumber\\
\lambda &= \int_0^{x^\ast} q(z)(1-e^{-z})^2dz - 2\int_{x^\ast}^\infty q(z)dz,\\
b &= \Big(1 + 2\int_{x^\ast}^\infty q(z)dz +2e^{x^\ast}\int_0^{x^\ast} q(z)dz\Big)e\nonumber
\end{align}
We remark that in \cite{MengersenTweedie} the Authors considered drift functions of the form $x\mapsto \exp(\mu |x|)$ for $0<\mu < \rho$ (recall that we consider the case $\rho = 2$). Our modification (the choice of $\mu = 1$ and multiplication by $e$) aims at simplifying the formulas we are about to obtain from Proposition \ref{prop:geometric_erg_new}.

Note that for $x^\ast$ sufficiently large $\lambda > 0$, moreover for every $x^\ast$, $b < \infty$. Note also that for a specific choice of $q$, the values of $\lambda, b$ can be estimated or computed (as is the case e.g. for $q(x) = 2^{-1}\exp(-|x|)$, a choice which is very convenient for simulations due to explicit formulas on the distribution function of the one sided exponential distribution).

Assume now that $g\colon \R \to \R$ is a function of polynomial growth, more specifically that for some constants $s > 0$ and $\kappa$,
\begin{align}\label{eq:growthcondition_new}
|g(x)| \le \kappa (1 + |x|)^s.
\end{align}
Since for $\alpha = 1/(s+1)$ we have $(1-\alpha)/\alpha = s$ and $ \log V(x) = (1+|x|)$, we get
\begin{displaymath}
|g(x)| \le \kappa\Big(\log  V(x)\Big)^{\frac{1-\alpha}{\alpha}}.
\end{displaymath}

Finally set $f = g - \pi g$. Then by Proposition \ref{prop:geometric_erg_new} we obtain that for $\alpha = 1/(s+1)$, and the chain starting from $0$ we obtain
\begin{align*}
\ccA\le& \ccA_{drift} = \kappa\frac{1}{\log \frac{1}{1-\lambda}}\frac{\log(b(1-\lambda)^{-1} + e^{x^\ast+1})}{\log 2} \times\\
&\phantom{aaaaaaaa}\times \Big[\Big(\frac{\log(e\lambda^{-1}+b\lambda^{-1})}{\log 2}\Big)^{\frac{s}{s+1}} + \frac{1}{(\log 2)^{s/(s+1)}} (\pi|g|)^{1/(s+1)}\Big]^{s+1},\\
\ccB\le& \ccB_{drift} = \kappa\frac{1}{\log \frac{1}{1-\lambda}}\frac{\log(\pi V + b(1-\lambda)^{-1}+e^{x^\ast+1})}{\log 2}\times\\
&\phantom{aaaaaaaa}\times\Big[\Big(\frac{\log(\pi V\lambda^{-1}+b\lambda^{-1})}{\log 2}\Big)^{\frac{s}{s+1}}+\frac{1}{(\log 2)^{s/(s+1)}} (\pi|g|)^{1/(s+1)}\Big]^{s+1},\\
\ccC,\ccD  \le& \ccC_{drift}=\kappa\frac{1}{\log \frac{1}{1-\lambda}}\frac{\log(b(1-\lambda)^{-1}+e^{x^\ast+1})}{\log 2}\times\\
&\phantom{aaaaaaaa}\times\Big[\Big(\frac{\log(b\lambda^{-1}+e^{x^\ast+1}\lambda^{-1})}{\log 2}\Big)^{\frac{s}{s+1}} +  \frac{1}{(\log 2)^{s/(s+1)}} (\pi|g|)^{1/(s+1)}\Big]^{s+1}.
\end{align*}

We also get an estimate on $\pi|g|$ and $\pi V$, namely
\begin{displaymath}
\pi|g| \le \left\{
\begin{array}{l}
\kappa \Big(\log(b\pi(C)\lambda^{-1})\Big)^{s}\quad{\rm for}\;s\le 1,\\
\kappa \Big(\log( b\pi(C)\lambda^{-1}) + (s-1)\Big)^{s} - (s-1)^{s}\quad{\rm for}\; s > 1.
\end{array}
\right.
\end{displaymath}
and
\begin{displaymath}
\pi V \le b\pi(C)\lambda^{-1}.
\end{displaymath}

Now Proposition \ref{pros1} together with the inequality $\ccA_{drift} \le \ccC_{drift}$ imply
\begin{align}\label{eq:MH1_new}
\cca,\ccc\le& (2r)^{s+1}\ccC_{drift} \le \kappa\Big(2\frac{\log(\frac{6}{2-\delta})}{\log(\frac{2}{2-\delta})}\Big)^{s+1}\ccC_{drift}\\
\ccb\le& r^{s+1}\Big(\max\{\ccB_{drift},\ccC_{drift}\}^{\frac{1}{s+1}}+\ccC_{drift}^{\frac{1}{s+1}}\Big)^{s+1} \nonumber\\
\le& \kappa\Big(\frac{\log(\frac{6}{2-\delta})}{\log(\frac{2}{2-\delta})}\Big)^{s+1}\Big(\max\{\ccB_{drift},\ccC_{drift}\}^{\frac{1}{s+1}}+\ccC_{drift}^{\frac{1}{s+1}}\Big)^{s+1},\nonumber
\end{align}
where $r$ is defined by \eqref{wrona0}.

Note that the values of the parameters appearing in the above estimates are directly computable from \eqref{eq:MHdrift_new}, the only exception being $\delta$ (given by \eqref{eq:form_of_delta_new}), the estimation of which may require some additional knowledge concerning the density $\pi$ (other than the ratios of $\pi(x)/\pi(y)$ for $x,y \in \R$). Note also that the parameter $\lambda$ improves (increases) as we increase $x^\ast$, contrary to the parameters $b$ and $\delta$. This leaves some room for rather nontrivial optimization in $x^\ast$, which however would be specific to a given form of $\pi$ and $q$.

\section{The main tool. Exponential inequalities in the independent case}\label{sect2}

In this section we develop some inequalities for sums of
independent (or one dependent) unbounded random variables, which
we will later combine with the renewal approach to obtain results
for Markov chains. Let us remark that there is a vast literature on concentration inequalities for sums of i.i.d. unbounded random variables with finite $\psi_\alpha$ norms for $\alpha \le 1$ (see in particular a series of papers by Borovkov \cite{Bor1,Bor2}). However, the tail inequalities, provided by those estimates quite often involve constants depending on the distribution of the underlying sequence, or when the constants depend only on $\alpha$, also the constant in front of the sub-gaussian coefficient is not universal (see e.g. inequality (1.4.) in \cite{MPR} and note that it cannot hold with $c_1$ being a universal constant multiplied by the variance and $c_2 \simeq \|X_i\|_{\psi_\gamma}$, as is witnessed e.g. by considering symmetric variables with $\P(X_i = \pm 1) = 1/n$ and $\P(X_i = 0) = 1 - 2/n$, $i=1,\ldots,n$, $n\to \infty$). For this reason, while such inequalities often perform better than ours for very large $t$, they are not well suited for the applications we have in mind. We do believe that inequalities which would work in our framework could be obtained by following the proofs from the aforementioned papers, but since the arguments we really need reduce to truncation and Bernstein type bounds on the Laplace transform, we prefer to provide complete proofs. Let us also stress that the results of this section are developed solely as technical tools for proving the corresponding results in the Markov chain setting and we do not claim any novelty of the methods we use.

The first lemma will let us truncate the
variables and reduce the problem to well known inequalities for
bounded summands. Contrary to the approach in \cite{Rad} at this
point we do not need Talagrand's inequalities and thus we are able
to obtain explicit constants (this happens at the cost of
weakening the generality of the result in the independent case,
but improves the inequalities which are then applied to Markov
chains). To avoid formal problems below we adopt the convention `$\log n := \log (n\vee e) $'.

\begin{lema}\label{unbounded_part} Let $\xi_0,\ldots,\xi_{n-1}$ be i.i.d. random
variables such that $\E\exp(c^{-\alpha}|\xi_i|^\alpha) \le 2$ for
some $\alpha \in (0,1]$. Let $M = c(3\alpha^{-2}\log n)^{1/\alpha}$ and $Y_i = \xi_i 1_{|\xi_i|>M}$ .
Then for any $0\le \lambda \le 1/(2^{1/\alpha}c)$,
\begin{displaymath}
\E\exp\Big(\lambda^\alpha \sum_{i=0}^{n-1}(|Y_i|+
\E|Y_i|)^\alpha\Big) \le \exp(8).
\end{displaymath}
\end{lema}

\begin{proof}
Note that by independence
\begin{align}\label{nierr0}
\E\exp\Big(\lambda^\alpha \sum_{i=0}^{n-1}(|Y_i|^\alpha +
(\E|Y_i|)^\alpha)\Big)= \exp\Big(n\lambda^\alpha (\E|\xi_0|
1_{|\xi_0|>M})^\alpha\Big)\Big(\E\exp(\lambda^\alpha
|\xi_0|^\alpha 1_{|\xi_0|>M})\Big)^{n}.
\end{align}
By Markov's inequality
\be\label{nier1} \P(|\xi_0|>t)\ls
e^{-\frac{t^{\alpha}}{c^{\alpha}}}\E\exp(\frac{|\xi_i|^{\alpha}}{c^{\alpha}})\ls
2e^{-\frac{t^{\alpha}}{c^{\alpha}}}. \ee

We thus obtain that \ba
&& \E|\xi_i|1_{|\xi_i|>M}= \int^{\infty}_0\P(|\xi_i|1_{|\xi_i|>M}>t)dt=M P(|\xi_i|>M)+\int^{\infty}_M \P(|\xi_i|>t)dt\ls \nonumber\\
\label{nier2} && \ls 2Me^{-\frac{M^{\alpha}}{c^{\alpha}}} +
2\int^{\infty}_{M}e^{-\frac{t^{\alpha}}{c^{\alpha}}}dt. \ea To
bound the last term in (\ref{nier2}) we need the following lemma.
\bl\label{lem1} If $M\gs 2^{1/\alpha}\alpha^{-\frac{1}{\alpha}}c$ then
$$
\int^{\infty}_{M} \exp(-\frac{t^{\alpha}}{c^{\alpha}})dt\ls
Me^{-\frac{M^{\alpha}}{c^{\alpha}}}.
$$
\el
\begin{proof}
The change of variables $s=t^{\alpha}/c^{\alpha}$ implies that
$$
\int^{\infty}_{M}e^{-\frac{t^{\alpha}}{c^{\alpha}}}dt=\frac{c}{\alpha}\int^{\infty}_{\frac{M^{\alpha}}{c^{\alpha}}}s^{\frac{1}{\alpha}-1}e^{-s}ds.
$$
Then again by the change of variables we deduce that
$$
\int^{\infty}_{M}e^{-\frac{t^{\alpha}}{c^{\alpha}}}dt=\frac{c}{\alpha}e^{-\frac{M^{\alpha}}{c^{\alpha}}}
\int^{\infty}_0(\frac{M^{\alpha}}{c^{\alpha}}+s)^{\frac{1}{\alpha}-1}e^{-s}ds.
$$
Using the inequality $1+x\ls \exp(x)$ we have
\begin{align*}
&\int^{\infty}_0(\frac{M^{\alpha}}{c^{\alpha}}+s)^{\frac{1}{\alpha}-1}e^{-s}ds=\frac{M^{1-\alpha}}{c^{1-\alpha}}\int^{\infty}_0 (1+\frac{c^{\alpha}}{M^{\alpha}}s)^{\frac{1}{\alpha}-1}e^{-s}ds\ls\\
&=\frac{M^{1-\alpha}}{c^{1-\alpha}}\int^{\infty}_0
\exp(-s(1-(\frac{1}{\alpha}-1)\frac{c^{\alpha}}{M^{\alpha}}))ds=
\frac{M^{1-\alpha}}{c^{1-\alpha}}(1-(\frac{1}{\alpha}-1)\frac{c^{\alpha}}{M^{\alpha}})^{-1}.
\end{align*}
Since by assumption  $M^{\alpha}/c^{\alpha}\ge 2/\alpha$, the right hand side above is bounded by
$M\alpha/c$ and thus
$$
\int^{\infty}_{M}\exp(-\frac{t^{\alpha}}{c^{\alpha}})dt\ls
M\exp(-\frac{M^{\alpha}}{c^{\alpha}}).
$$
\end{proof}
Plugging the estimate from Lemma \ref{lem1} into (\ref{nier2}) we
conclude that
$$
\exp(\lambda^{\alpha}(\E |\xi_0|1_{|\xi_0|>M})^{\alpha})\ls
\exp(\lambda^{\alpha}[4 M
e^{-\frac{M^{\alpha}}{c^{\alpha}}}]^{\alpha}).
$$
By definition $M=c(3\alpha^{-2}\log n)^{1/\alpha} \ge c(2\alpha^{-2}\log n + \alpha^{-2})^{1/\alpha}$,
therefore
$$
\exp(n\lambda^{\alpha}(\E |\xi_0|1_{|\xi_0|>M})^{\alpha})\ls
\exp(3\cdot 4^\alpha n\lambda^\alpha c^\alpha\alpha^{-2}\frac{\log n}{n^{2\alpha^{-1}}}e^{-\alpha^{-1}} ) \le
\exp(3\cdot 4^\alpha \lambda^\alpha c^\alpha \frac{4}{e^2}\frac{\log n}{n} ),
$$
where we used the fact that $\sup_{\alpha \in (0,1)}\alpha^{-2}\exp(-\alpha^{-1}) = 4/e^2$.
Since $e^2 \ge 7$, we see that whenever $\lambda^\alpha c^\alpha \le 1/2$, we have
\begin{align}\label{nier3}
\exp(n\lambda^{\alpha}(\E |\xi_0|1_{|\xi_0|>M})^{\alpha}) \le \exp(4).
\end{align}

Now we proceed to the estimate
on $\E\exp(\lambda^{\alpha}|\xi_0|^{\alpha} 1_{|\xi_0|> M}|)$.

There is the following representation

$$
\E\exp(\lambda^{\alpha}
|\xi_0|^{\alpha}1_{|\xi_0|>M})=1+\int^{\infty}_{1}
\P(|\xi_0|1_{|\xi_0|>M}>\lambda^{-1}\log^{\frac{1}{\alpha}} t)dt.
$$
Let $t_0=\exp(\lambda^{\alpha} M^{\alpha})$, i.e.
$M=\lambda^{-1}\log^{\frac{1}{\alpha}}t_0$. Observe that one can
split the integral bound into two terms
$$
\int^{\infty}_{1}\P(|\xi_0|1_{|\xi_0|>M}>\lambda^{-1}\log^{\frac{1}{\alpha}}
t))dt=
(t_0-1)\P(|\xi_0|>M)+\int^{\infty}_{t_0}\P(|\xi_0|>\lambda^{-1}\log^{\frac{1}{\alpha}}
t)dt.
$$
Applying (\ref{nier1}) again we obtain that
\begin{align*}
& \int^{\infty}_{1}\P(|\xi_0|1_{|\xi_0|>M}>\lambda^{-1}\log^{\frac{1}{\alpha}} t))dt\ls \\
&\ls 2t_0 e^{-\frac{M^{\alpha}}{c^{\alpha}}}+2\int^{\infty}_{t_0} t^{-\lambda^{-\alpha} c^{-\alpha}}dt\ls \\
&\ls 2 t_0 e^{-\frac{M^{\alpha}}{c^{\alpha}}}+
2\frac{\lambda^{\alpha} c^{\alpha}}{1-\lambda^{\alpha}
c^{\alpha}}t_0^{1-\lambda^{-\alpha} c^{-\alpha}}.
\end{align*}
Since $t_0=\exp(\lambda^{\alpha}{M^{\alpha}})$  we have
$t_0e^{-\frac{M^{\alpha}}{c^{\alpha}}}=t_0^{1-\lambda^{-\alpha}c^{-\alpha}}$
and thus we conclude that whenever $\lambda c <1$,
$$
\E\exp(\lambda^{\alpha} |\xi_0|^{\alpha}1_{|\xi_0|>M})\ls
1+\frac{2}{1-\lambda^{\alpha} c^{\alpha}}\exp(-(1-\lambda^{\alpha}
c^{\alpha})\frac{M^{\alpha}}{c^{\alpha}}).
$$
Thus requiring that $\lambda^{\alpha} c^{\alpha}\ls  1/2$ we
guarantee that
$$
\E\exp(\lambda^{\alpha} |\xi_0|^{\alpha}1_{|X_i|>M})\ls
1+4e^{-\frac{M^{\alpha}}{2c^{\alpha}}}.
$$
Since $M\ge c(2\log n)^{\frac{1}{\alpha}}$ we
deduce
$$
\E\exp(\lambda^{\alpha} |\xi_0|^{\alpha}1_{|\xi_0|>M})\ls
1+\frac{4}{n}.
$$
This gives \be(\E\exp(\lambda^\alpha
|X_1|^{\alpha}1_{|X_1|>M}))^{n}\ls (1+\frac{4}{n})^{n}\ls \exp(4),
\ee which together with (\ref{nierr0}) and (\ref{nier3}) ends the
proof.
\end{proof}

\begin{lema}\label{Laplace_bounded}
Let $\xi_0,\ldots,\xi_{n-1}$ be i.i.d. mean zero random variables
such that $|\xi_i|\ls M$ for $i\in\{0,...,n-1\}$ with variance $\sigma^2$ and let $T\le {n}$ be a
stopping time (with respect to some filtration $\mathcal{F}_i
\supseteq \sigma(\xi_0,\ldots,\xi_{i-1})$ such that $\xi_i$ is
independent of $\mathcal{F}_i$). Then for every $a > 0$,
$\varepsilon \in (0,1)$ and
\be\label{mocny}
\lambda \le \min\Big(\frac{3}{2M(1+\varepsilon)},
\frac{\sqrt{\varepsilon}}{(1+\varepsilon)\sigma\sqrt{\|(T-a)_+\|_{\psi_1}}}\Big),
\ee
we have
\begin{displaymath}
\E\exp(|\lambda \sum_{i=1}^{T} \xi_{i-1}|) \le
2^{1+\varepsilon/(1+\varepsilon)}\exp\Big(\frac{\lambda^2
a(1+\varepsilon)\sigma^2}{2(1 - \lambda(1+\varepsilon) M/3)}\Big)
\end{displaymath}
\end{lema}
\begin{proof}
Let us consider the martingale
\begin{displaymath}
M_k = \frac{\exp(\lambda(1+\varepsilon)\sum_{i=1}^k
\xi_{i-1})}{(\E\exp(\lambda(1+\varepsilon) \xi_0))^k}.
\end{displaymath}
By Doob's theorem $\E M_T = 1$ and thus by the H\"older inequality
\begin{align*}
\E\exp(\lambda\sum_{i=1}^T\xi_{i-1}) &= \E
\exp(\lambda\sum_{i=1}^T\xi_{i-1})(\E\exp(\lambda(1+\varepsilon)\xi_0))^{-T/(1+\varepsilon)}
(\E\exp(\lambda (1+\varepsilon)\xi_0))^{T/(1+\varepsilon)}\\
&\le (\E M_T)^{1/(1+\varepsilon)}\Big(\E (\E\exp(\lambda
(1+\varepsilon)\xi_0))^{T/\varepsilon}\Big)^{\varepsilon/(1+\varepsilon)}\\
&=\Big(\E (\E\exp(\lambda
(1+\varepsilon)\xi_0))^{T/\varepsilon}\Big)^{\varepsilon/(1+\varepsilon)}.
\end{align*}
The classical Bernstein (e.g. \cite{Bou}) bound gives
\begin{displaymath}
\E\exp(\lambda (1+\varepsilon)\xi_0) \le
\exp\Big(\frac{\lambda^2(1+\varepsilon)^2\sigma^2}{2(1-\lambda(1+\varepsilon)M/3)}\Big)
\end{displaymath}
and thus
\begin{align*}
\E\exp(\lambda\sum_{i=1}^T\xi_{i-1}) &\le \Big(\E
\exp\Big(\frac{\lambda^2
(1+\varepsilon)^2\sigma^2T}{2\varepsilon(1-\lambda(1+\varepsilon)M/3)}\Big)\Big)^{\varepsilon/(1+\varepsilon)}\\
&=
\exp\Big(\frac{\lambda^2(1+\varepsilon)\sigma^2a}{2(1-\lambda(1+\varepsilon)M/3)}\Big)\Big(\E
\exp\Big(\frac{\lambda^2
(1+\varepsilon)^2\sigma^2(T-a)_+}{2\varepsilon(1-\lambda(1+\varepsilon)M/3)}\Big)\Big)^{\varepsilon/(1+\varepsilon)}.
\end{align*}
Thus if
\be\label{slaby}
\lambda\ls \frac{3}{M(1+\va)}(\frac{M^2\va}{9\|(T-a)_{+}\|_{\psi_1}\sigma^2}
(-1+\sqrt{1+\frac{18\sigma^2\|(T-a)_{+}\|_{\psi_1}}{M^2\va}}))
\ee
then the inequality in question holds. Note that (\ref{mocny}) implies (\ref{slaby}) (one can also directly see that (\ref{mocny}) implies the assertion of the lemma).
\end{proof}

The next proposition gives a bound on the tail of sums of
independent or one-dependent random variables with finite
$\psi_\alpha$ norms. Its second part may seem to be formulated in
a somewhat artificial way, however when dealing with the Markov
case this formulation will help us introduce the asymptotic
variance of the additive functional to the inequalities.

\begin{prop}\label{thm1}
(i) Let $\xi_0,\ldots,\xi_{n-1}$ be a one dependent sequence of
mean zero, variance $\sigma^2$ random variables, such that $\E
\exp(c^{-\alpha}|\xi_i|^\alpha) \le 2$ and let $m$ be a positive
integer. Let $M=c(3\alpha^{-2}\log(n/m))^{1/\alpha}$,
then for any $t \ge 0$,
\begin{displaymath}
\P(\sup_{k < n/m}|\sum^{k}_{i=1} \xi_{i-1} |>t)\ls 2e^{8}e^{-\frac{t^{\alpha}}{2(4c)^{\alpha}}} + 4
\exp(-\frac{t^2}{32(\lceil n/(2m)\rceil\sigma^2+Mt/6)}).
\end{displaymath}

\noindent (ii)Assume additionally that the variables $\xi_i$ are
independent and let $N$ be a stopping time (wrt some filtration
$\mathcal{F}_i \supseteq \sigma(\xi_0,\ldots,\xi_{i-1})$ such that
$\xi_i$ is independent of $\mathcal{F}_i$) such that $N \le n$. Let  $M=c(3\alpha^{-2}\log n)^{1/\alpha}$. Then for any $\varepsilon \in (0,1)$, $a >0$
and $p,q > 0$ such that $p^{-1}+q^{-1}= 1$, we have for all $t \ge
0$,
\begin{displaymath}
\P(|\sum^{N}_{i=1} \xi_{i-1} |>t)\ls
e^{8}e^{-\frac{(p^{-1}t)^{\alpha}}{2c^{\alpha}}} +
2^{1+\varepsilon/(1+\varepsilon)}\exp\Big(-\frac{q^{-2}t^2}{2((1+\varepsilon)a\sigma^2+\mu^{-1}
tq^{-1})}\Big),
\end{displaymath}
where
\begin{displaymath}
\mu = \min\Big(\frac{3}{4M(1+\varepsilon)},
\frac{\sqrt{\varepsilon}}{(1+\varepsilon)\sigma\sqrt{\|(N-a)_+\|_{\psi_1}}}\Big).
\end{displaymath}

\end{prop}
\begin{proof}
First observe that the case of one-dependent random variables can be easily transformed to the question of independent ones. Indeed it suffices to split the sum into odd and even part, namely we use
\be\label{radon1}
\P(\sup_{k<n/m}|\sum^k_{i=1} \xi_{i-1}|>t)\ls
\P(\sup_{k<n/m}|\sum^{\infty}_{i=0}\xi_{2i}1_{2i<k}|>t/2)+\P(\sup_{k<n/m}|\sum^{\infty}_{i=0}\xi_{2i+1}1_{2i+1<k}|>t/2).
\ee
Then decompose with respect to $M$, i.e. let $\xi_i=\xi_i 1_{|\xi_i|>M}+\xi_i1_{|\xi_i|\ls M}$ and denote
$$
Y_i=\xi_i1_{|\xi_i|>M}-\E\xi_{i}1_{|\xi_i|>M},\;\; Z_i=\xi_i
1_{|\xi_i|\le M}-\E \xi_i1_{|\xi_i|\le M}.
$$
It results in the following inequality
\be\label{nier0}
\P(\sup_{k<n/m}|\sum^{\infty}_{i=0}\xi_{2i}1_{2i<k}|>t/2) \ls \P(\sup_{k<n/m}|\sum^{\infty}_{i=0}Y_{2i}1_{2i<k}|>t/4)+\P(\sup_{k<n/m}
|\sum^{\infty}_{i=0}Z_{2i}1_{2i < k}|>t/4).
\ee
We use
the usual Laplace transform argument on each of the summands. We
have by Markov's and triangle inequalities (recall that $\alpha
\in (0,1]$, so $|x+y|^\alpha \le |x|^\alpha+|y|^\alpha$),
\begin{displaymath}
\P(\sup_{k<n/m}|\sum^{\infty}_{i=0}Y_{2i}1_{2i<
k}|>t/4)\ls
e^{-\lambda^{\alpha}t^{\alpha}/4^{\alpha}}
\E\exp(\lambda^{\alpha}\sum^{\infty}_{i=0}|Y_{2i}|^{\alpha}1_{2i<(n-1)/m}).
\end{displaymath}
By Lemma \ref{unbounded_part}, for $\lambda \le 1/(2^{1/\alpha}c)$,
$$
\E\exp(\lambda^{\alpha}\sum^{\infty}_{i=0}|Y_{2i}|^{\alpha}1_{2i<
(n-1)/m}) \le \exp(8),
$$
which gives
\begin{align*}
& e^{-\frac{\lambda^{\alpha}t^{\alpha}}{4^{\alpha}}}\E
\exp(\lambda^{\alpha}|\sum^{\infty}_{i=0}Y_{2i}1_{2i<
(n-1)/m}|^{\alpha})\ls \exp(8)
e^{-\frac{t^{\alpha}}{2(4c)^{\alpha}}}.
\end{align*}
This shows that the unbounded part of the sum is well
concentrated.
\smallskip

\noindent Let us now pass to the bounded part. Similarly as before, we have
for all $\bar{\mu}\gs 0$
$$
\P(\sup_{k<n/m}|\sum^{\infty}_{i=0}Z_{2i}1_{2i <
k}|>t/4)\ls e^{-\bar{\mu} t/4 }\E
\exp(\bar{\mu}|\sum^{\infty}_{i=0}Z_{2i}1_{2i< (n-1)/m}|),
$$
where we have used Doob's maximal inequality \cite{Yor} together with the
fact that the sequence $M_k = \exp(\bar{\mu}|\sum_{i=0}^\infty Z_{2i}1_{2i< k}|)$ is
a submartingale.
\smallskip

\noindent
Since $|Z_0| \le 2M$, we can now use the classical Bernstein bound (again see \cite{Bou})
to get for $0 \le \bar{\mu} < 3/(2M)$,
\begin{displaymath}
\E \exp(\bar{\mu}|\sum^{\infty}_{i=0}Z_{2i}1_{2i< (n-1)/m}|) \le
2\exp(\frac{\bar{\mu}^2\sigma^2 \lceil n/2m\rceil}{2(1-2M\bar{\mu} /3)}),
\end{displaymath}
optimizing in $\bar{\mu}\gs 0$ we deduce that
\begin{displaymath}
\P(\sup_{k<n/m}|\sum^{\infty}_{i=0}Z_{2i}1_{2i <
k}|>t/4)\ls 2\exp(-\frac{t^2}{32(\lceil n/(2m)\rceil\sigma^2 +
Mt/6)}).
\end{displaymath}
Combining this inequality with the previous estimate on the
unbounded part gives
\begin{displaymath}
\P(\sup_{k< n/m}|\sum^{\infty}_{i=0}\xi_{2i}1_{2i\ls k}|>t/2) \le
e^{8}e^{-\frac{t^{\alpha}}{2(4c)^{\alpha}}} + 2
\exp(-\frac{t^2}{32(\lceil n/(2m)\rceil\sigma^2+Mt/6)}).
\end{displaymath}
Using an analogous argument for
$\P(\sup_{k<n/m}|\sum_{i=0}^\infty\xi_{2i+1}1_{2i+1<k}| > t/2)$ and then (\ref{radon1}) we
finally derive the inequality asserted in part (i) of the
proposition.
\smallskip

\noindent
Let us now consider part (ii). We will use the above notation with
obvious modifications, which we will not state explicitly (the
difference stems from the fact that now there is no need to split
the sum into the 'odd' and 'even' parts). Instead of the symmetric split we apply
any $p,q\gs 1$ such that $p^{-1}+q^{-1}=1$ and then
\be\label{nierow0}
\P(|\sum^{N}_{i=1}\xi_{i-1}|>t) \ls \P(|\sum^{N}_{i=1}Y_{i-1}|>p^{-1}t)+\P(
|\sum^{N}_{i=1}Z_{i-1}|>q^{-1}t).
\ee
The unbounded part can be handled in the same way
as for part (i). As for the bounded part,
from Lemma \ref{Laplace_bounded},
for any $a\ge 0$, $\varepsilon \in (0,1)$ and
\begin{displaymath}
\mu \le \mu_{1} = \min\Big(\frac{3}{4M(1+\varepsilon)},
\frac{\sqrt{\varepsilon}}{(1+\varepsilon)\sigma\sqrt{\|(N-a)_+\|_{\psi_1}}}\Big),
\end{displaymath}
we get
\begin{displaymath}
\E\exp(|\mu \sum_{i=1}^{N} Z_{i-1}|) \le
2^{1+\varepsilon/(1+\varepsilon)}\exp\Big(\frac{\mu^2
a(1+\varepsilon)\sigma^2}{2(1 - \mu\mu_{1}^{-1})}\Big).
\end{displaymath}
Thus
$$
e^{-q^{-1}\mu t }\E \exp(\mu|\sum_{i=1}^{N} Z_{i-1}|)\ls
e^{-q^{-1}\mu
t}2^{1+\varepsilon/(1+\varepsilon)}\exp\Big(\frac{a(1+\varepsilon)\mu^2\sigma^2}{2(1-\mu\mu_{1}^{-1})}\Big),
$$
from which we deduce like in Bernstein's inequality (setting $\mu
= q^{-1}t/((1+\varepsilon)a\sigma^2+\mu_{1}^{-1}q^{-1}t)$)
$$
\P(|\sum_{i=1}^{N} Z_{i-1}| > q^{-1}t) \ls
2^{1+\varepsilon/(1+\varepsilon)}\exp(-\frac{q^{-2}t^2}{2((1+\varepsilon)a\sigma^2+\mu_{1}^{-1}
tq^{-1})}),
$$
which ends the proof of part (ii).
\end{proof}
\paragraph{Remark}
We note that modifying the martingale argument in part (ii) of Proposition \ref{thm1}
one can show a similar inequality under the assumptions of part (i).

\medskip

Let us now pass to corresponding results in the case of suprema of
empirical processes. We will consider a countable class $\mathcal{F}$ of
functions $f \colon
\mathcal{X} \to \R$ (countability is important only for measurability
purposes and clearly can be relaxed). Let $F(x) = \sup_{f\in\mathcal{F}} |f(x)|$ be
the envelope of $\mathcal{F}$. Our goal will be to obtain
exponential bounds for $\P(S_\mathcal{F}^\ast \ge (1+\varepsilon)\E S_\mathcal{F} + t)$, where
\begin{displaymath}
S_{\mathcal{F}} = \Big\|\sum_{i=1}^n f(\xi_{i-1})\Big\|_\mathcal{F} := \sup_{f\in \mathcal{F}} \Big|\sum_{i=1}^n f(\xi_{i-1})\Big|
\end{displaymath}
and
\begin{displaymath}
S_{\mathcal{F}}^\ast = \max_{k\le n}\sup_{f\in \mathcal{F}} \Big|\sum_{i=1}^k
f(\xi_{i-1})\Big|
\end{displaymath}
($\xi_i$ are independent random variables).

Our main tool will be the following version of Talagrand's
inequality, obtained by Klein and Rio \cite{KleinRio} in the case
of bounded summands.

\begin{theo}\label{KleinRio}
Let $\xi_0,\ldots,\xi_{n-1}$ be i.i.d. random variables and assume that $\mathcal{F}$ is a countable class of functions with an envelope $F$, such that $\E f(\xi_0) =0$ for all $f \in\mathcal{F}$ and
$|F(\xi_i)| \le M$ a.s. Then for any $\lambda \le 2/(3M)$,
\begin{displaymath}
\E\exp(\lambda S) \le \exp\Big(\lambda \E S + \frac{(2M\E S
+n\sigma^2)\lambda^2}{2-3M\lambda}\Big),
\end{displaymath}
where $\sigma^2 = \sup_{f\in\mathcal{F}} \E f(\xi_0)^2$.
\end{theo}

Similarly as in \cite{EL} (formula (3.2)) by using the fact that
$\exp(\sup_f|\sum_{i=1}^k f(\xi_{i-1})|)$ is a submartingale, Doob's
maximal inequality and Bernstein's approach we obtain the following

\begin{coro} \label{KleinRio_cor}In the setting of Theorem \ref{KleinRio}, for any $t\ge 0$,
\begin{displaymath}
\P(S_\mathcal{F}^\ast \ge \E S_\mathcal{F} +t) \le \exp\Big(-\frac{t^2}{2\sigma^2n + (4\E S_\mathcal{F} +
3t)M}\Big)
\end{displaymath}
and in consequence for any $\varepsilon > 0$,
\begin{displaymath}
\P(S_\mathcal{F}^\ast \ge (1+\varepsilon)\E S_\mathcal{F} +t) \le
\exp\Big(-\frac{t^2}{2(1+\varepsilon)n\sigma^2}\Big) +
\exp\Big(-\frac{t}{   MD_\varepsilon}\Big),
\end{displaymath}
where $D_\varepsilon = (1+\varepsilon^{-1})(3+4\varepsilon^{-1})$.
\end{coro}

\paragraph{Remark}
Two important aspects of the above inequalities are that the subgaussian behavior of the tail estimate for small $t$ is governed by the weak-variance $\sigma^2$ and that the parameter $\varepsilon$ can be taken arbitrarily small, which shows in particular that for Donsker classes $\mathcal{F}$ the concentration properties of the empirical process are almost as good as for the limiting Gaussian
process.

We remark that there are other recent inequalities for suprema of empirical processes, which take into account some other parameters of the process than we consider here and also assume some additional knowledge on the so called bracketing numbers of the class $\mathcal{F}$ (see e.g. \cite{LedvdGeer}, where the weak variance is replaced by a Bernstein type upper bound for moments of individual functions $f$). In concrete application their advantage is that they provide also an upper bound on $\E S_\mathcal{F}$, however at the same time they are not as general as Talagrand's inequality. Our `philosophy' here is that we would like to separate two aspects of the study of empirical processes, i.e. bounds on the expectation (which are intimately related to concentration properties for individual functions $f$ and are usually obtained via chaining methods) and concentration for the supremum (which we would like to express, at least at the subgaussian level, via the weak variance of the process). This is also the reason why we prefer to introduce the additional parameter $\varepsilon$ into the inequalities, since then the right estimate on probability does not involve the expectation of the supremum any more and in particular the growth with $n$ of the exponents becomes explicit (note that while the expectation is always of the order at most $\mathcal{O}(n)$, in typical situations, i.e. under some additional geometric assumptions it is smaller). To the best of our knowledge eliminating the factor $(1+\varepsilon)$ in front of $\E S_\mathcal{F}$, without introducing additional parameters (i.e. other than $\sigma^2$ and $M$) to the probability bound remains an open problem.

\paragraph{} Our next goal is to give tools, which will allow to prove similar inequalities for suprema of additive functionals of Markov chains. Combining Corollary \ref{KleinRio_cor} with Lemma \ref{unbounded_part} we
obtain
\begin{prop}\label{unbounded_emp_ind} Consider i.i.d. random variables $\xi_0,\ldots,\xi_{n-1}$ and assume that $\mathcal{F}$ is a countable class of functions with an envelope $F$, such that $\E f(\xi_0) =0$ for all $f \in\mathcal{F}$ and
$\E \exp(c^{-\alpha}F(\xi_i)^\alpha) \le 2$.  Let  $M = c(3\alpha^{-2}\log n)^{1/\alpha}$ and $\sigma^2 = \sup_{f\in\mathcal{F}} \E f(\xi_0)^2$. Then for
any $\varepsilon \in (0,1/2)$ and
$t\ge 0$,
\begin{align*}
\P(S_\mathcal{F}^\ast \ge (1+\varepsilon)\E S_\mathcal{F} +t) \le& \exp\Big(-\frac{(1-2\varepsilon)^2 t^2}{2(1+\varepsilon)n\sigma^2}\Big) +
e\exp\Big(-\frac{t(1-2\varepsilon)}{2M(1+\varepsilon^{-1})(3+4\varepsilon^{-1})}\Big) \\
&+ e^8\exp\Big({-\frac{(\varepsilon t)^\alpha}{2c^\alpha}}\Big).
\end{align*}
\end{prop}

\begin{proof}  For $f \in \mathcal{F}$, define
the functions $f_1(x) = f(x)1_{|F(x)| \le M} - \E f(\xi_0)1_{|F(\xi_0)| \le M}$, $f_2(x) = f(x) -
f_1(x) = f(x)1_{|F(x)|
> M} - \E f(\xi_0)1_{|F(\xi_0)|
> M}$. Let $\mathcal{F}_i = \{f_i\colon f\in \mathcal{F}\}$, $i =
1,2$. Clearly $S_\mathcal{F}^\ast \le S_{\mathcal{F}_1}^\ast + S_{\mathcal{F}_2}^\ast$, and thus

\begin{align*}
\P(S_\mathcal{F}^\ast \ge (1+\varepsilon)\E S_\mathcal{F} + t) \le \P(S_{\mathcal{F}_1}^\ast \ge(1+\varepsilon)\E S_\mathcal{F} + (1-\varepsilon)t) +
\P(S_{\mathcal{F}_2}^\ast \ge \varepsilon t).
\end{align*}

We have
\begin{displaymath}
S_{\mathcal{F}_2}^\ast \le \sum_{i=1}^n\Big(F(\xi_{i-1})1_{|F(\xi_{i-1})|
> M} + \E F(\xi_{i-1})1_{|F(\xi_{i-1})|> M}\Big)
\end{displaymath}
and so by Lemma \ref{unbounded_part} and Chebyshev's inequality we get
\begin{align}
\P(S_{\mathcal{F}_2}^\ast \ge\varepsilon t) \le e^8e^{-\frac{(\varepsilon t)^\alpha}{2c^\alpha}}.
\end{align}

Let us note that without loss of generality we can assume that $t \ge 16M\varepsilon^{-1}$ (otherwise the right hand side of the inequality in question exceeds one). Therefore

\begin{align*}
(1+\varepsilon) \E S_{\mathcal{F}} &\ge (1+\varepsilon) \E S_{\mathcal{F}_1} - 2\E S_{\mathcal{F}_2}
\ge (1+\varepsilon) \E S_{\mathcal{F}_1} - 4n\E F(\xi_0) 1_{F(\xi_0) > M}\\
&\ge (1+\varepsilon) \E S_{\mathcal{F}_1} - 16Mn\exp(-M^\alpha/c^\alpha) \ge (1+\varepsilon) \E S_{\mathcal{F}_1} - 16M \\
&\ge (1+\varepsilon) \E S_{\mathcal{F}_1} - \varepsilon t,
\end{align*}
where the third inequality follows from Lemma \ref{lem1}.

Taking into account that for every $f\in \mathcal{F}$, $\|f_1\|_\infty \le 2M$ and $\E f_1(\xi_0)^2 \le \E f(\xi_0)^2 $, we get by Corollary \ref{KleinRio_cor},
\begin{align*}
\P(S_{\mathcal{F}_1}^\ast > (1+\varepsilon)\E S_\mathcal{F} + (1-\varepsilon)t) &\le
\P(S_{\mathcal{F}_1}^\ast > (1+\varepsilon)\E S_{\mathcal{F}_1} + (1-2\varepsilon)t)\\
&\le \exp\Big(-\frac{(1-2\varepsilon)^2 t^2}{2(1+\varepsilon)n\sigma^2}\Big) +
\exp\Big(-\frac{t(1-2\varepsilon)}{2MD_\varepsilon}\Big),
\end{align*}
which ends the proof of the proposition.
\end{proof}

\section{Exponential concentration for additive functionals of Markov chains}\label{sect3}

In this section we will prove tail estimates for Markov chains
expressed in terms of quantities introduced in Section
\ref{sect1}.

\subsection{Additive functionals}
We will start from the most general of our inequalities and later we
will add assumptions under which we are able to improve some
aspects of the estimates.

The inequalities we present are expressed in terms of the parameters $\cca,\ccb,\ccc$ introduced in Section \ref{sect1}, formula (\ref{definition_abc}). Our discussion in Section \ref{sect1} shows that there exist bounds on  $\cca,\ccb,\ccc$ in terms of $\ccA,\ccB,\ccC,\ccD$ (Proposition \ref{pros1}), which in turn can be
estimated via drift conditions (Theorems \ref{thmcond1},
\ref{thmcond2}). We do not plug the most general version of those inequalities into our tail estimates so as not to obscure the already
quite involved formulas. At the end of the section we will do this only in the case of geometrically ergodic Markov chains to obtain Theorem \ref{thm:A_new} from the Introduction with explicit formulas for the parameters involved.

Recall that $\pi^{\ast}$ is the split of $\pi$ measure.

\bt\label{thm2} Let $\{\bf X\}$ be an ergodic Markov chain and
$\{\bf X^m\}$ its $m$-skeleton used in the split chain
construction. Assume for simplicity that $m|n$ and let $f\colon \mathcal{X} \to \R$ be an arbitrary function for which the parameters $\cca,\ccb,\ccc$ defined in \eqref{definition_abc} are finite and such that $\E_\pi f = 0$. Set $M=\ccc(3\alpha^{-2}\log(n/m))^{1/\alpha}$. Then the following
inequality holds for all $t \ge 0$,
\begin{align*}
 \P_x(|\sum^{n-1}_{i=0}f(X_i)|>3t)\ls& 2\exp\Big(-\frac{t^{\alpha}}{\cca^{\alpha}}\Big)+ 2\pi^{\ast}(\theta)^{-1}\exp\Big(-\frac{t^{\alpha}}{\ccb^{\alpha}}\Big)+
2e^{8}\exp\Big(-\frac{t^{\alpha}}{2(4\ccc)^{\alpha}}\Big) \\
&+ 4
\exp\Big(-\frac{t^2}{32(\lceil n/(2m)\rceil\sigma^2+Mt/6)}\Big).
\end{align*}
where $\cca,\ccb,\ccc$ have been defined by formula
(\ref{definition_abc}) and $\sigma^2 = \E s_0^2$.
 \et

\begin{proof}
We have introduced all necessary tools to prove the exponential concentration. By the construction of the split chain,
$$
\P_x(|\sum^{n-1}_{i=0}f(X_i)|>3t)=\bar{\P}_{x^{\ast}}(|\sum^{n-1}_{i=1}f(X_i)|>3t).
$$
In (\ref{decomp}) we have decomposed $\sum^{n-1}_{i=0}f(X_i)$ into three summands
$$
|\sum^{n-1}_{k=0}f(X_k)|\le U_n(f)+V_n(f)+W_n(f),
$$
therefore to complete our proof it suffices to bound
$$
\bar{\P}_{\ast}(|U_n|>t),\;\;\bar{\P}_{\ast}(|V_n|>t),\;\;\bar{\P}_{\ast}(|W_n|>t).
$$

\noindent
For simplicity we have assumed that $m|n$, then all the quantities can be expressed in terms of
$Z_k(f)$, $k\gs 0$, namely
$$
U_n(f)=\Big|\sum^{\min(\sigma,n/m-1)}_{k=0}Z_k(f)\Big|,\;\;V_n(f)=\sum^{N}_{i=1}s_{i-1},\;\;W_n(f)=\Big|1_{N > 0}\sum^{\sigma(N)}_{k=n/m}Z_k(f)\Big|.
$$
To bound the first term note that \be\label{tg1}
\bar{\P}_{x^{\ast}}(|U_n(f)|>t)\ls
2\exp(-\frac{t^{\alpha}}{\cca^{\alpha}}), \ee where
$\cca=\|\sum^{\sigma}_{k=0}|Z_k(f)|\|_{\psi_{\alpha},\bar{\P}_{x^\ast}}$.
Estimating $W_n(f)$ is more involved. By Pitman's occupation measure formula (\cite{Pit1,Pit2}, see also \cite[Theorem 10.0.1]{MeynTwe}) we get
$$
\sum^{\infty}_{l=1}\bar{\E}_{\theta}\exp(\ccb^{-\alpha}|\sum^{\tau_{\alpha}}_{k=l} Z_k(f)|^{\alpha})1_{\tau_\alpha \ge l}=\pi^{\ast}(\theta)^{-1}\E_{\pi^{\ast}}
 \exp(\ccb^{-\alpha}|\sum^{\sigma(0)}_{k=0} Z_k(f)|^{\alpha}),
$$
and it follows that
\begin{align*}
&\bar{\P}_{x^{\ast}}(|W_n(f)|>t)\ls \sum^{n/m-1}_{l=1}\bar{\P}_{x^{\ast}}(\bar{X}^m_{n/m-l}\in \theta)\bar{\P}_{\theta}(|\sum^{\tau_{\alpha}}_{k=l}Z_k(f)|>t\;\&\;\tau_{\alpha}\gs l)\ls  \\
&\ls \sum^{n/m-1}_{l=1}\bar{\P}_{x^{\ast}}(\bar{X}^m_{n/m-l}\in \theta)e^{-\frac{t^{\alpha}}{\ccb^{\alpha}}}
\bar{\E}_{\theta}(1_{\tau_{\alpha}\gs l}\exp(\ccb^{-\alpha}|\sum^{\tau_{\alpha}}_{k=l}Z_k(f)|^{\alpha})\ls\\
&\ls [\max_{1\ls l <n/m}\bar{\P}_{x^{\ast}}(\bar{X}^m_{n/m-l}\in \theta)]e^{-\frac{t^{\alpha}}{\ccb^{\alpha}}}\pi^{\ast}(\theta)^{-1}\E_{\pi^{\ast}}
 \exp(\ccb^{-\alpha}(\sum^{\sigma(0)}_{k=0}|Z_k(f)|)^{\alpha}).
\end{align*}
Obviously
$$
[\max_{1\ls l <n/m}\bar{\P}_{x^{\ast}}(\bar{X}^m_{n/m-l}\in \theta)]\ls 1,
$$
but we stress here that $\lim_{n\ra\infty}\bar{\P}_{x^{\ast}}(\bar{X}^m_{n/m-l}\in \theta)=\pi^{\ast}(\theta)$, so the bound
usually can be much better if we wait for ergodicity to be observed. However in the general case it shows that
$$
\bar{\P}_{x^{\ast}}(W_n(f)>t)\ls e^{-\frac{t^{\alpha}}{\ccb^{\alpha}}}[\pi^{\ast}(\theta)^{-1}\E_{\pi^{\ast}}
 \exp(\ccb^{-\alpha}(\sum^{\sigma(0)}_{k=0}|Z_k(f)|)^{\alpha}) ],
$$
so due to the definition of $\ccb$ it yields \be\label{tg2}
\bar{\P}_{x^{\ast}}(W_n(f)>t)\ls
2\pi^{\ast}(\theta)^{-1}e^{-\frac{t^{\alpha}}{\ccb^{\alpha}}}. \ee

Finally we should give a bound on $\bar{\P}_{x^{\ast}}(V_n>t)$, yet this
is exactly the setting of our main tool. If we set $\xi_i=s_i$ and
$$
c=\|s_i(f)\|_{\psi_{\alpha},\bar{\P}}=\ccc,
$$
then Proposition \ref{thm1} gives
$$
\bar{\P}_{x^{\ast}}(|V_n |>t)\ls 2e^{8}e^{-\frac{t^{\alpha}}{2(4c)^{\alpha}}}
+ 4
\exp(-\frac{t^2}{32(\lceil n/(2m)\rceil\sigma^2+Mt/6)}),
$$
where $\sigma^2=\bar{\E} s_1^2$.
\end{proof}

One of important features of the classical Bernstein inequality is
that for small $t$ (i.e. $t \ll n$) it provides a subgaussian tail
estimate with the subgaussian coefficients given by the variance
of the i.i.d. summands. Thus it shows that for some range of $t$
the tail estimate is almost as good as for the limiting Gaussian
variable.

Note that for $m>1$ the subgaussian coefficients from Theorem
\ref{thm2} differs from the asymptotic variance (i.e. the variance of the limiting Gaussian distribution), which as is well
known (see eg. \cite{Che,MeynTwe}) is equal to $\pi^\ast(\theta)(\E s_0(f)^2 + 2\E
s_0(f)s_1(f))$. Moreover also in the case $m=1$, the subgaussian
coefficients remains bounded away from the asymptotic variance, which in this case is equal to $\pi^\ast(\theta)\E s_0(f)^2$.

We would now like to provide an estimate for additive functionals
with the subgaussian coefficients arbitrarily close to
$\pi^\ast(\theta)\E s_0(f)^2$. To achieve this we will have to make additional
assumptions concerning the Markov chain. Namely, we will assume
that it is geometrically ergodic, i.e. that
$\|\sigma(0)+1\|_{\psi_1,\P_{\nu^\ast}} = \|\sigma(1)-\sigma(0)\|_{\psi_1} < \infty$ (recall that $\nu$ is the
minorizing measure used in the splitting construction) and that
$m=1$ (i.e. the chain is strongly aperiodic).

Our proof will rely on the second part of Proposition \ref{thm1}.
In order to use it we will first show that the stopping time $N$
does not deviate much from $n\pi^\ast(\theta)$.
The following result is classical and can be found e.g. in \cite{Bou}.
\begin{lema}\label{le:Bernstein} [Bernstein's $\psi_1$ inequality] If
$\xi_0,\ldots,\xi_{n-1}$ are independent mean zero random
variables such that $\|\xi_i\|_{\psi_1} \le c$, then for all
$t\ge 0$,
\begin{displaymath}
\P(\sum_{i=0}^{n-1}\xi_i\ge t) \le \exp(-\frac{t^2}{4nc^2 +
2ct}).
\end{displaymath}
\end{lema}

Before we formulate the next lemma, let us introduce one more parameter, which quantifies geometric ergodicity
of the chain. Namely, define

\begin{align}\label{d_def}
{\bf d} = \|\sigma(1)-\sigma(0)\|_{\psi_1}  = \|\sigma(0)+1\|_{\psi_1,\P_{\nu^\ast}}< \infty
\end{align}
(note that this definition does not depend on the initial distribution of the chain).

As is well known geometric ergodicity is equivalent to finiteness of ${\bf d}$, which can be effectively bounded by using classical drift conditions (see e.g. \cite{MeynTwe}). More precisely we have the following

\begin{prop}\label{prop:boundond_new}
If the chain satisfies \eqref{small}  with $m=1$ and the drift condition \eqref{eq:geometric_ergodicity_drift_new}, then
\begin{align*}
{\bf d} &\le 2r\max\Big(\frac{\log(b(1-\lambda)^{-1}+K)}{\log 2},1\Big)\frac{1}{\log \frac{1}{1-\lambda}}\\
&\le 2\Big(\frac{\log(\frac{6}{2-\delta})}{\log(\frac{2}{2-\delta})}\Big)\max\Big(\frac{\log(b(1-\lambda)^{-1}+K)}{\log 2},1\Big)\frac{1}{\log \frac{1}{1-\lambda}},
\end{align*}
where $r$ is defined by \eqref{wrona0}.
\end{prop}

\begin{proof}The proposition follows from the third estimate in Lemma \ref{le:auxiliary_regeneration_new} and the bound on $\ccc$ given in Proposition \ref{pros1} applied to $\alpha=1$ and the function $f \equiv 1$.
\end{proof}

Recall the stopping time $N$ defined in (\ref{N_def}) and note that by the law of large numbers, for $n \gg 1$ it should behave like $\pi^\ast(\theta)n$. The next lemma quantifies this intuition and gives a bound on deviations of $N$.

\begin{lema}\label{N_integr}
Assume that $m=1$ and the Markov chain is geometrically ergodic.
 Then for any $\varepsilon \in (0,1)$ and every integer $k \ge \pi^\ast(\theta)n(1+\varepsilon)$,
\begin{align*}
\bar{\P}_{x^\ast}(N > k) \le  \exp\Big(-\frac{k-\pi^\ast(\theta)n}{36\pi^\ast(\theta)^2{\bf d}^2\varepsilon^{-1}}\Big).
\end{align*}
In consequence,
\begin{displaymath}
\|(N - (1+\varepsilon)\pi^\ast(\theta)n)_+\|_{\psi_1,\bar{\P}_{x^\ast}} \le 144 \pi^\ast(\theta)^{2}{\bf
d}^2\varepsilon^{-1}.
\end{displaymath}
\end{lema}

\begin{proof}
Set $T_i = \sigma(i)- \sigma(i-1)$ for $i \ge 1$ and note that
$T_i$ are i.i.d. random variables with common mean $\E T_i =
\pi^\ast(\theta)^{-1} \le {\bf d}$. We have $\|T_i - \E T_i\|_{\psi_1} \le
2{\bf d}$. Thus for every integer
$k = \pi^\ast(\theta)n(1+\varepsilon + t)$, $t \ge 0$,

\begin{align*}
\bar{\P}_{x^\ast}(N > k) &= \bar{\P}_{x^\ast}(\sigma(k) < n - 1) \le \bar{\P}_{x^\ast}(\sum_{i=1}^{k}
T_i  < n-1) \\
&\le \bar{\P}_{x^\ast}(\sum_{i=1}^{k} (T_i - \pi^\ast(\theta)^{-1}) \le n-1 -
\pi^\ast(\theta)^{-1}k)  \nonumber\\&= \bar{\P}_{x^\ast}(\sum_{i=1}^{k} (T_i -
\pi^\ast(\theta)^{-1} ) \le -n(\varepsilon + t) - 1)\nonumber\\
&\le \exp(-\frac{n^2(\varepsilon +t)^2}{16{\bf
d}^2\pi^\ast(\theta)n(1+\varepsilon + t) + 4(\varepsilon +t)n{\bf d}}),
\end{align*}
where we used Lemma \ref{le:Bernstein}.

Taking into account that $\pi^\ast(\theta){\bf d} \ge 1$ and $(\varepsilon + t + 1)/(\varepsilon + t) \le 2/\varepsilon$, this yields
\begin{align*}
\bar{\P}_{x^\ast}(N > k) \le \exp\Big(-\frac{\pi^\ast(\theta) n(\varepsilon +t)}{36\pi^\ast(\theta)^2{\bf d}^2\varepsilon^{-1}}\Big)
= \exp\Big(-\frac{k-\pi^\ast(\theta)n}{36\pi^\ast(\theta)^2{\bf d}^2\varepsilon^{-1}}\Big).
\end{align*}
Thus for $Z = (N- (1+\varepsilon)\pi^\ast(\theta)n)_+$ and $a= 36\pi^\ast(\theta)^2{\bf d}^2\varepsilon^{-1}$, we have
\begin{align*}
\E \exp\Big(\frac{Z}{2a}\Big) =& 1 + \int_0^\infty e^t\bar{\P}_{x^\ast}(Z > 2at)dt =
 1+\int_0^{1/(2a)} e^tdt +\int_{1/(2a)}^\infty\exp(t - (2at -1)/a)dt \\
&= e^{1/(2a)} + e^{1/a}e^{-1/(2a)}= 2e^{1/(2a)} \le 4,
\end{align*}
which proves the Lemma.
\end{proof}

Repeating the proof of Theorem \ref{thm2} and using the second
part of Proposition \ref{thm1} instead of the first one, together with the above Lemma \ref{N_integr}, we get the following theorem. Note that by playing with the parameters $p,\varepsilon$ we can now make the subgaussian coefficient arbitrarily close to the optimal one (given by the CLT) at the cost of worsening the remaining constants.

\begin{theo}\label{thm:geometric_ergodicity_old}
Let $\{\bf X\}$ be a strongly aperiodic, geometrically ergodic Markov chain and
$\{\bf X^1\}$ its split version. Let $f\colon \mathcal{X} \to \R$ be a function for which the parameters $\cca,\ccb,\ccc$, defined in \eqref{definition_abc} are finite and such that $\E_\pi f = 0$. Let $p,q > 1$ satisfy $p^{-1}
+q^{-1} = 1$ and $\varepsilon \in (0,1)$. Then the following
inequality holds for all $t \ge 0$,
\begin{align*}
& \P_x(|\sum^{n-1}_{i=0}f(X_i)|>t)\ls
2\exp(-\frac{(tp^{-1})^{\alpha}}{(2\cca)^{\alpha}})+
2\pi^{\ast}(\theta)^{-1}\exp(-\frac{(tp^{-1})^{\alpha}}{(2\ccb)^{\alpha}})\\
&+ e^{8}\exp(-\frac{(p^{-1}q^{-1}t)^{\alpha}}{2\ccc^{\alpha}}) +
2^{1+\varepsilon/(1+\varepsilon)}\exp(-\frac{q^{-4}t^2}{2((1+\varepsilon)\sigma^2n+M(\varepsilon)
tq^{-2})}),
\end{align*}
\noindent where
\begin{displaymath}
\sigma^2 = \sigma^2(f) = \pi^\ast(\theta)\E s_0(f)^2
\end{displaymath}
and
\begin{displaymath}
M(\varepsilon) = \max\Big(\frac{4c(3\alpha^{-2}\log n)^{1/\alpha}(1+\varepsilon)}{3},
\frac{12\pi^\ast(\theta){\bf d}(1+\varepsilon)\sigma}{\varepsilon}\Big).
\end{displaymath}
\end{theo}

\begin{theo}\label{thm:main_geometric_new}
Assume that $\{X_n\}_{n\ge0}$ is a Harris recurrent strongly aperiodic Markov chain on the space $(\mathcal{X},\mathcal{B})$ admitting a unique stationary measure $\pi$. Assume furthermore that conditions \eqref{small} with $m=1$ and \eqref{eq:geometric_ergodicity_drift_new} are satisfied.
Let finally $s > 0$ and consider an arbitrary measurable function $g\colon \mathcal{X} \to \R$ such that
\begin{displaymath}
|g(x)| \le \kappa \Big(\log(V(x))\Big)^s
\end{displaymath}
for some $\kappa \ge 0$. Set $\alpha = 1/(s+1)$. Then for every $x \in \mathcal{X}$, $\eta \in (0,1]$ and $t > 0$,
\begin{align}\label{eq:main_est_geometric_erg_new_1}
& \P_x(|\sum^{n-1}_{i=0}g(X_i) - n\pi g|>t)\ls
2\exp(-\frac{(t\eta)^{\alpha}}{(25\cca)^{\alpha}})+
2\pi^{\ast}(\theta)^{-1}\exp(-\frac{(t\eta)^{\alpha}}{(25\ccb)^{\alpha}})\\
&+ e^{8}\exp(-\frac{(\eta t)^{\alpha}}{2\cdot 14^\alpha\ccc^{\alpha}}) +
2^{1+\eta/(2+\eta)}\exp(-\frac{t^2}{2((1+\eta)\sigma^2n+M(\eta)
t)}),\nonumber
\end{align}
\noindent where
\begin{align}\label{eq:sigmadef_new}
\sigma^2 &= \pi^\ast(\theta)\E s_0(g - \pi g)^2 \\
&= \lim_{n\to \infty} \frac{\Var(\sum_{i=0}^{n-1} g(X_n))}{n} = \E_\pi (g(X_0)-\pi g)^2 + 2\sum_{i=1}^\infty \E_\pi (g(X_0)-\pi g)(g(X_i)-\pi g),
\nonumber
\end{align}
and
\begin{displaymath}
M(\eta) = (1+\eta)^{3/4} \max\Big(\frac{4\ccc(3\alpha^{-2}\log n)^{1/\alpha}}{3},
\frac{29\pi^\ast(\theta){\bf d}\sigma}{\eta}\Big).
\end{displaymath}
and the numbers $\cca,\ccb,\ccc,\mathbf{d}$ satisfy
\begin{align*}
& \cca\ls  r^{1/\alpha}\Big((\max\{\ccA_{drift},\ccC_{drift}\})^{\alpha}+\ccC_{drift}^{\alpha}\Big)^{\frac{1}{\alpha}},\\
& \ccb \le r^{1/\alpha} \Big((\max\{\ccB_{drift},\ccC_{drift}\})^{\alpha}+\ccC_{drift}^{\alpha}\Big)^{\frac{1}{\alpha}},\\
& \ccc \le (2r)^{1/\alpha} \ccC_{drift},\\
& \mathbf{d} \le 2r\max\Big(\frac{\log(b(1-\lambda)^{-1}+K)}{\log 2},1\Big)\frac{1}{\log \frac{1}{1-\lambda}},
\end{align*}
where
\begin{displaymath}
r \le \frac{\log(\frac{6}{2-\delta})}{\log(\frac{2}{2-\delta})}
\end{displaymath} is the unique solution to the equation
\begin{align}
2^{1/r}\delta^{1-1/r} + 2^{1+1/r}(1-\delta)^{1-1/r} = 2,
\end{align}

\begin{align*}
\ccA_{drift} &= \kappa \frac{1}{\log \frac{1}{1-\lambda}} \max\Big(\frac{\log(V(x)1_{C^c}(x) + (b(1-\lambda)^{-1} + K)1_C(x))}{\log 2},1\Big) \times\\
&\phantom{aaa}\times \Big[\Big(\max(\frac{\log(V(x)\lambda^{-1}+b\lambda^{-1})}{\log 2},1)\Big)^{1-\alpha} + \frac{1}{(\log 2)^{1-\alpha}}(\pi|g|/\kappa)^\alpha\Big]^{1/\alpha},\\
\ccB_{drift} &= \kappa\frac{1}{\log \frac{1}{1-\lambda}} \max\Big(\frac{\log(\pi V + (b(1-\lambda)^{-1}+K)\pi(C))}{\log 2},1\Big)\times\\
&\phantom{aaa}\times\Big[\Big(\max(\frac{\log(\pi V\lambda^{-1}+b\lambda^{-1})}{\log 2},1)\Big)^{1-\alpha} + \frac{1}{(\log 2)^{1-\alpha}}(\pi|g|/\kappa)^\alpha\Big]^{1/\alpha},\\
\ccC_{drift} &= \kappa\frac{1}{\log \frac{1}{1-\lambda}}\max\Big(\frac{\log(b(1-\lambda)^{-1}+K)}{\log 2},1\Big)\times\\
&\phantom{aaa}\times \Big[\Big(\max(\frac{\log(b\lambda^{-1}+K\lambda^{-1})}{\log 2},1)\Big)^{1-\alpha} + \frac{1}{(\log 2)^{1-\alpha}}(\pi|g|/\kappa)^\alpha\Big]^{1/\alpha}.
\end{align*}
Moreover
\begin{displaymath}
\pi|g| \le \left\{
\begin{array}{l}
\kappa \Big(\log(b\pi(C)\lambda^{-1})\Big)^{s}\quad{\rm for}\;s \le 1,\\
\kappa \Big(\log(b\pi(C)\lambda^{-1}) + s-1\Big)^{s} - (s-1)^s \quad{\rm for}\; s>1.
\end{array}
\right.
\end{displaymath}
and
\begin{displaymath}
\pi V \le b\pi(C)\lambda^{-1}.
\end{displaymath}
\end{theo}

\begin{proof}[Proof of Theorem \ref{thm:main_geometric_new}]
One simply uses Theorem \ref{thm:geometric_ergodicity_old} with $f = g - \pi g$ and $q^4 = 1+\varepsilon = \sqrt{1+\eta}$ and performs some elementary calculations to obtain the dependence on $\eta$ in the inequalities.
To bound the parameters $\cca,\ccb,\ccc$ as well as $\pi V, \pi |g|$ one uses Propositions \ref{pros1}, whereas the bound on ${\bf d}$ follows from Proposition \ref{prop:boundond_new}.

The second and third inequality in \eqref{eq:sigmadef_new} is a well known fact concerning the asymptotic variance, see e.g. \cite{MeynTwe} or \cite{Che}.
\end{proof}

\paragraph{Example continued} The above bounds apply in particular to the algorithmic examples discussed in Section \ref{secRegularDrift}, where we estimated all the parameters except for ${\bf d}$ (for functions of polynomial growth), which however can be directly estimated by Proposition \ref{prop:boundond_new}. We remark that the constants obtained in this examples, even in the simplest cases (e.g. for the geometric distribution with parameter $1/2$ or for the Gaussian density and exponential proposal) are rather bad.  In the example with the geometric distribution, the value of parameters are for $s = 1$ of the order $10^3$, in the continuous example there are much worse, since they are additionally multiplied by the factor corresponding to $\delta$. It is of practical interest to derive bounds with a better dependence on $\delta$ and $\alpha$, however we are not aware of any previous results which would give exponential inequalities with any explicit constants.

Let us also mention that the parameter $\sigma$, which as we already explained in the Introduction, is important from the theoretical point of view, in practical situations is unknown and has to be estimated in terms of the drift. One approach would be to use the bound
\begin{align}\label{eq:boundonsigma}
\|s_0(f)\|_2 \le 2\alpha^{-1/2}\Gamma(2/\alpha)^{1/2} \|s_0(f)\|_{\psi_\alpha} = 2\alpha^{-1/2}\Gamma(2/\alpha)^{1/2} \ccc,
\end{align}
 which follows easily from integration by parts and combine it with estimates on $\ccc$ from Theorem \ref{thm:main_geometric_new}. This approach has however the disadvantage that the bound is bad for $\alpha$ small. Some other estimates, for functions bounded by some small power of $V$ are given in \cite{LaMieNie}. Their advantage stems from the fact that they do not go through the $\psi_\alpha$ norm and thus do not involve dependence on $\alpha$. On the other hand the dependence on $V$ is not logarithmic like in our case, but polynomial (which is optimal under the assumptions of \cite{LaMieNie}).

\begin{proof}[Proof of Theorem \ref{thm:A_new}]
All the parameters of the inequality from Theorem \ref{thm:main_geometric_new} are bounded in terms of the drift and $\delta$, the only exception being $\sigma$ and $\pi^\ast(\theta)^{-1} = \E_{\nu^\ast} \sigma(0)+1 \le {\bf d}$ and $M$ (which depends on ${\bf d}$ and $\sigma$). Thus the theorem follows from Proposition \ref{prop:boundond_new} and \eqref{eq:boundonsigma}.
\end{proof}

\subsection{Empirical processes\label{sec:empirical_MArkov_new}}

In this section we present estimates for empirical processes of geometrically ergodic Markov chains.
The interest in this type of random variables stems from their wide applicability in statistics and machine learning theory, e.g. model selection, M-estimation or other statistical methods based on minimization of some empirical criteria.

We want to obtain concentration inequalities for chains started from arbitrary initial conditions, so in addition to the parameter ${\bf d}$ we will introduce its counterpart, measuring how quickly the chain regenerates for a given starting point
\begin{displaymath}
{\bf e} = \|\sigma(0)\|_{\psi_1,\bar{\P}_{x^\ast}}.
\end{displaymath}
We will need this parameter since in the empirical process case we are not able to use the martingale argument of the second part of Proposition \ref{thm1} and so to obtain the inequalities with the subgaussian term close to the optimal one, we will have to consider separately the case of large and small $N$.

We remark that it can be bounded from above similarly as the parameter ${\bf d}$, namely we have

\begin{prop}\label{prop:boundone_new}
If the chain satisfies \eqref{small}  with $m=1$ and the drift condition \eqref{eq:geometric_ergodicity_drift_new}, then
\begin{align*}
{\bf e} \le &r\Big(\max\Big(\frac{\log(V(x)1_{C^c}(x) + (b(1-\lambda)^{-1}+K)1_C(x))}{\log 2},\frac{\log(b(1-\lambda)^{-1}+K)}{\log 2},1\Big)
+1\Big)\frac{1}{\log \frac{1}{1-\lambda}},
\end{align*}
where $r$ is given by equation \eqref{wrona0}.
\end{prop}

\begin{proof}The proposition follows from the first and third estimate in Lemma \ref{le:auxiliary_regeneration_new} and the bound on $\cca$ given in Proposition \ref{pros1} applied to $\alpha=1$ and the function $f \equiv 1$.
\end{proof}

Our main result concerning empirical processes of Markov chains is the following

\begin{theo}\label{Markov_emp_thm}
Let $\{\bf X\}$ be a strongly aperiodic geometrically ergodic Markov chain and
$\{\bf X^1\}$ its split version. Let $\mathcal{F}$ be a countable
class of $\pi$-centered functions $f\colon \mathcal{X} \to \R$ with an envelope
$F$ and assume that for some $\alpha \in (0,1]$ the parameters
$\cca, \ccb, \ccc$ for the function $F$ are finite.

Denote
\begin{displaymath}
Z= \sup_{f\in\mathcal{F}}|\sum_{i=0}^{n-1}f(X_i)|.
\end{displaymath}

Let $\varepsilon \in (0,1/2)$ and denote $M=\ccc(3\alpha^{-2}\log(n))^{1/\alpha}$ and
\begin{displaymath}
\sigma^2 = \pi^\ast(\theta)\sup_{f\in\mathcal{F}} \E s_0(f)^2.
\end{displaymath}

Then
\begin{align*}
\bar{\P}_{x^\ast}\Big(Z \ge (1+7\varepsilon) \E Z + t\Big) &\le \exp\Big(-\frac{(1-2\varepsilon)^4 t^2}{2(1+\varepsilon)^2n\sigma^2}\Big) +
e\exp\Big(-\frac{t(1-2\varepsilon)^2}{2M(1+\varepsilon^{-1})(3+4\varepsilon^{-1})}\Big) \\
&+ e^8\exp\Big({-\frac{(\varepsilon (1-2\varepsilon)t)^\alpha}{2\ccc^\alpha}}\Big)
+e\exp\Big(-\frac{\varepsilon^2 n}{144{\pi^\ast(\theta)\bf d}^2}\Big)\\
& + 2\exp\Big(-\frac{(\varepsilon t)^\alpha}{(2\cca)^\alpha}\Big)
+2\pi^\ast(\theta)^{-1}\exp\Big(-\frac{(\varepsilon t)^\alpha}{(2\ccb)^\alpha}\Big).
\end{align*}
for all $t \ge C(\varepsilon)$, where
\begin{displaymath}
C(\varepsilon) = \varepsilon^{-1}\Big(\Big(9 + 9{\bf e}
+ 27\pi^\ast(\theta){\bf d}^2\varepsilon^{-1}\Big)\pi(F) \\
+ 9\Gamma(1+1/\alpha) \cca
+ 9\cdot 2^{1/\alpha-1}e\Gamma(1+1/\alpha) \log^{1/\alpha}(e/\pi^\ast(\theta))\ccb\Big).
\end{displaymath}
\end{theo}

\paragraph{Remark} Let us now discuss certain aspects of the above rather technical (and not very friendly looking) theorem, which may help understand its applicability and limitations.

\paragraph{1.} Let us point out that the involved form of the estimate is a result of the effort to obtain a probability bound which for small $t$ would exhibit the subgaussian behavior with the subgaussian parameter arbitrarily close to the weak variance (which for Donsker classes of functions is responsible for the concentration of the limiting Gaussian distribution). Clearly by choosing $\varepsilon$ small enough one can approach the right subgaussian coefficient arbitrarily close at the cost of worsening the constants in the other terms and decreasing the interval for which the estimate behaves like in the Gaussian case (the length of this interval is of the order $n^{1/(2-\alpha)}$ for $\alpha < 1$ and $n/\log n$ for $\alpha = 1$, which is greater then the CLT scaling $\sqrt{n}$).

\paragraph{2.} Second, let us note that the fourth term of the estimate does not involve $t$, but depends only on $n$. This is a consequence of our method and we do not know if one could remove this term completely without worsening the dependence of the parameters $\cca,\ldots,{\bf d}$ in the other terms. It is relatively easy to replace this term e.g. by $\exp(-\varepsilon^2 t/(C\pi^\ast(\theta){\bf d}^2\ccc))$, where $C$ is a universal constant, however, since standard applications of this type of tail estimates involve $t$ of order at most $n$, for which the 'troublesome' term is dominated by terms involving $t$, we do not pursue this direction here (see \cite{MarkovOrlicz} where inequalities of this type are proven in a slightly different context).

\paragraph{3.} Finally, we would like to comment on the threshold $C(\varepsilon)$ appearing in the estimates. It is of order $\varepsilon^{-2}$ and is independent of $n$, so it does not pose a problem in typical applications (when $t$ is of order $n^\gamma$). Moreover, even without this threshold the estimate would not yield any information for $t$ of order smaller then $\varepsilon^{-2}$, since the denominator of the exponent in the second term is of the order $\varepsilon^{-2}$. At present we do not know whether the dependence of the threshold and the estimates on $\varepsilon$ and parameters ${\bf d}$,${\bf e}$ and $\pi^\ast(\theta)$ can be improved (note that ${\bf d},{\bf e}$ are not related just to the function $F$ but are 'global' parameters of the chain). At the end of the article we will present an example to argue that the above theorem can give meaningful estimates even if we replace $1 + 7\varepsilon$ by an arbitrary constant.

\paragraph{4.} In particular the above theorem implies Theorem \ref{thm:B_new} from the Introduction. The dependence of the parameters under the drift condition, in the case of parameters $\cca,\ccb,\ccc,{\bf d}$ can be derived just as in Theorem \ref{thm:main_geometric_new}, while the parameter ${\bf e}$ has been estimated in Proposition \ref{prop:boundone_new}. We leave the details to the reader.

\medskip
Let us now pass to the proof of Theorem \ref{Markov_emp_thm}. The next lemma is a strengthening of an analogous estimate from  \cite{Rad} (in particular it incorporates the dependence on
$\varepsilon$ to the inequality). Its estimate will be used to compare the expectation of the random sum introduced with the regeneration method and the random variable $Z$.

\begin{lema}\label{expectation_estimate} Let $\{\bf X\}$ be a strongly aperiodic Harris ergodic Markov chain and
$\{\bf X^1\}$ its split version. Let $\mathcal{F}$ be a countable
class of functions $f\colon \mathcal{X} \to \R$ with an envelope
$F$. Then for every $\varepsilon\in (0,1/2)$,
\begin{align*}
\bar{\E}_{x^\ast} \Big\|\sum_{i=1}^{\lfloor (1+\varepsilon)\pi^\ast(\theta)n\rfloor} s_{i-1}(f)\Big\|_{\mathcal{F}} \le &(1+4\varepsilon)\Big(\bar{\E}_{x^\ast} \Big\|\sum_{i=1}^{N} s_{i-1}(f)\Big\|_{\mathcal{F}}\\
&+ \Big(2+ 2\bar{\E}_{x^\ast}  \sigma(0)
+ 3\frac{\pi^\ast(\theta){\bar{\E}_{x^\ast}  (\sigma(1)-\sigma(0))^2}}{\varepsilon}\Big)\pi(F)\Big).
\end{align*}
\end{lema}
\begin{proof}
Let $n_1= \lfloor (1-\varepsilon)\pi^\ast(\theta)n\rfloor$, $n_2 =
\lfloor (1+\varepsilon)\pi^\ast(\theta)n\rfloor$ and denote
$\mathcal{B} = \sigma(\sum_{i=1}^{n_2} s_{i-1}(f)\colon f
\in\mathcal{F})$. By Jensen's inequality and exchangeability of
random vectors $(s_i(f))_{f\in \mathcal{F}}$, $i=0,\ldots,n-1$, if $n_1 + 1 \le n_2$, we
have
\begin{align*}\label{expectation1}
\bar{\E}_{x^\ast}  s_0(F) + \bar{\E}_{x^\ast}  \Big\|\sum_{i=1}^{n_1} s_{i-1}(f)\Big\|_{\mathcal{F}} &\ge \bar{\E}_{x^\ast}  \Big\|\sum_{i=1}^{n_1+1} s_{i-1}(f)\Big\|_{\mathcal{F}}\\
&\ge \bar{\E}_{x^\ast}
\Big\|\sum_{i=1}^{n_1+1} \E(s_{i-1}(f)|\mathcal{B})
\Big\|_{\mathcal{F}} = \frac{n_1+1}{n_2}\bar{\E}_{x^\ast} \Big\|\sum_{i=1}^{n_2}
s_{i-1}(f) \Big\|_{\mathcal{F}}.\nonumber
\end{align*}
Combining the above inequality with the trivial case $n_1 = n_2$, we get
\begin{align}
\bar{\E}_{x^\ast}  \Big\|\sum_{i=1}^{n_2}
s_{i-1}(f) \Big\|_{\mathcal{F}} \le \max(\frac{n_2}{n_1+1},1)\Big(\bar{\E}_{x^\ast}  s_0(F) + \bar{\E}_{x^\ast}  \Big\|\sum_{i=1}^{n_1} s_{i-1}(f)\Big\|_{\mathcal{F}}\Big).
\end{align}
Moreover
\begin{align}\label{expectation2}
\bar{\E}_{x^\ast}  \Big\|\sum_{i=1}^{n_1} s_{i-1}(f)\Big\|_{\mathcal{F}} \le& \bar{\E}_{x^\ast}
\Big\|\sum_{i=1}^{n_1} s_{i-1}(f)\Big\|_{\mathcal{F}}1_{\{N\ge
n_1\}}+ \bar{\E}_{x^\ast}  \Big\|\sum_{i=1}^{N}
s_{i-1}(f)\Big\|_{\mathcal{F}}1_{\{N<n_1\}} \\
&+ \bar{\E}_{x^\ast}
\Big\|\sum_{i=N+1}^{n_1}
s_{i-1}(f)\Big\|_{\mathcal{F}}1_{\{N<n_1\}}.\nonumber
\end{align}
By Doob's optional sampling theorem,
\begin{align}\label{expectation3}
\bar{\E}_{x^\ast}
\Big\|\sum_{i=1}^{n_1} s_{i-1}(f)\Big\|_{\mathcal{F}}1_{\{N\ge
n_1\}} \le \bar{\E}_{x^\ast}  \Big\|\sum_{i=1}^{N} s_{i-1}(f)\Big\|_{\mathcal{F}}1_{\{N\ge
n_1\}}.
\end{align}
Moreover, by independence of $(s_i(f))_{f\in\mathcal{F},i\ge N}$ and
$N$, the third summand on the right hand side of (\ref{expectation2}) does not exceed
\begin{displaymath}
\bar{\E}_{x^\ast} (n_1-N)_+\bar{\E}_{x^\ast}  s_0(F) =  \pi^\ast(\theta)^{-1} \pi(F)\bar{\E}_{x^\ast} (n_1-N)_+
\end{displaymath}
To bound $\bar{\E}_{x^\ast} (n_1-N)_+$ we can proceed as in the proof of Lemma
\ref{N_integr}, however this time we do not need exponential inequalities. Set again $T_i = \sigma(i) -
\sigma(i-1)$ for $i \ge 1$ and additionally $T_0 = \sigma(0)$.

For $t\in (0,n_1/(\pi^\ast(\theta)n))$ we have

\begin{align*}
\bar{\P}_{x^\ast}((n_1-&N)_+ > \pi^\ast(\theta)nt) =
\bar{\P}_{x^\ast}(N< n_1-\pi^\ast(\theta)nt)\\
& \le \bar{\P}_{x^\ast}(\sigma(\lfloor n_1-\pi^\ast(\theta)nt\rfloor ) \ge n-1) = \bar{\P}_{x^\ast}\Big(\sum_{i=0}^{\lfloor n_1-\pi^\ast(\theta)nt\rfloor}
T_i \ge n -1\Big) \\
&= \bar{\P}_{x^\ast}(T_0 \ge t n/2-1) + \bar{\P}_{x^\ast}\Big(\sum_{i=1}^{\lfloor n_1-\pi^\ast(\theta)nt\rfloor} T_i \ge
(1-t/2)n\Big) \\
&\le \bar{\P}_{x^\ast}(2(T_0+1) \ge tn) +
\bar{\P}_{x^\ast}\Big(\sum_{i=1}^{\lfloor n_1-\pi^\ast(\theta)nt\rfloor}(T_i - \E T_i) \ge (1-t/2)n -
(1-\varepsilon - t)n\Big) \\
&\le \bar{\P}_{x^\ast}(2(T_0+1) \ge tn) + \frac{n_1\E T_1^2}{n^2(\varepsilon+t/2)^2}
\le \bar{\P}_{x^\ast}(2(T_0+1) \ge tn) + \frac{\pi^\ast(\theta)\E T_1^2}{n(\varepsilon+t/2)^2},
\end{align*}
where we used the fact that for $i > 0$, $\E T_i = \pi^\ast(\theta)^{-1}$

Now
\begin{align*}
\bar{\E}_{x^\ast}(n_1-N)_+ &= \pi^\ast(\theta)n \int_0^1\bar{\P}_{x^\ast}((n_1-N)_+ \ge \pi^\ast(\theta)nt)dt\\
&\le \pi^\ast(\theta)n\int_0^1 \Big(\bar{\P}_{x^\ast}(2(T_0+1) \ge tn) + \frac{\pi^\ast(\theta)\E T_1^2}{n(\varepsilon+t/2)^2}\Big)dt\\
&\le 2\pi^\ast(\theta)(\bar{\E}_{x^\ast} T_0+1) +2\varepsilon^{-1}\pi^\ast(\theta)^2\bar{\E}T_1^2.
\end{align*}

Thus
\begin{displaymath}
\bar{\E}_{x^\ast}
\Big\|\sum_{i=N+1}^{n_1}
s_{i-1}(f)\Big\|_{\mathcal{F}}1_{\{N<n_1\}} \le 2\Big(1+ \bar{\E}_{x^\ast} T_0
+ \frac{\pi^\ast(\theta){\bar{\E}T_1^2}}{\varepsilon}\Big)\pi(F).
\end{displaymath}
Combining the above estimate with (\ref{expectation1}), (\ref{expectation2}), (\ref{expectation3}) and the inequality
$$\bar{\E}_{x^\ast} s_0(F) = \pi^\ast(\theta)^{-1}\pi(F) \le \pi^\ast(\theta)\bar{\E}T_1^2\pi(F),$$
we get
\begin{displaymath}
\bar{\E}_{x^\ast} \Big\|\sum_{i=1}^{n_2} s_{i-1}(f)\Big\|_{\mathcal{F}} \le \max(\frac{n_2}{n_1+1},1)\Big(\bar{\E}_{x^\ast} \Big\|\sum_{i=1}^{N} s_{i-1}(f)\Big\|_{\mathcal{F}}
+ \Big(2+ 2\bar{\E}_{x^\ast} T_0
+ 3\frac{\pi^\ast(\theta){\bar{\E}T_1^2}}{\varepsilon}\Big)\pi(F)\Big).
\end{displaymath}
To finish the proof it suffices to note that for $\varepsilon \in (0,1/2)$,
\begin{displaymath}
\frac{n_2}{n_1+1} \le \frac{(1+\varepsilon)n\pi^\ast(\theta)}{(1 - \varepsilon) n\pi^\ast(\theta)}
\le 1 + 4\varepsilon.
\end{displaymath}
\end{proof}

\begin{coro} \label{expectation_estimate_cor}In the setting of Theorem \ref{Markov_emp_thm}, for every $\varepsilon \in (0,1/2)$,
\begin{align*}
\bar{\E}_{x^\ast} \Big\|\sum_{i=1}^{\lfloor (1+\varepsilon)\pi^\ast(\theta)n\rfloor} s_{i-1}(f)\Big\|_{\mathcal{F}} \le &(1+4\varepsilon)\Big(\bar{\E}_{x^\ast} \Big\|\sum_{i=1}^{N} s_{i-1}(f)\Big\|_{\mathcal{F}}
\\&+ \Big(2 + 2{\bf e}
+ 6\pi^\ast(\theta){\bf d}^2\varepsilon^{-1}\Big)\pi(F)\Big).
\end{align*}
\end{coro}

\begin{proof}
We use Lemma \ref{expectation_estimate} together with the inequality
\begin{displaymath}
1 + \frac{\E X}{\|X\|_{\psi_1}} + \frac{\E X^2}{2\|X\|_{\psi_1}^2} \le 2,
\end{displaymath}
which holds for every random variable with $\|X\|_{\psi_1} \neq 0$.
\end{proof}

\begin{lema}\label{UW_est} In the setting of Theorem \ref{Markov_emp_thm}
\begin{align*}
\bar{\E}_{x^\ast} U_n(F) \le 2\Gamma(1+1/\alpha) \cca,
\end{align*}
and
\begin{displaymath}
\bar{\E}_{x^\ast} W_n(F) \le 2^{1/\alpha}e\Gamma(1+1/\alpha) \log^{1/\alpha}(e/\pi^\ast(\theta))\ccb,
\end{displaymath}
where $\Gamma(z) = \int_0^\infty t^{z-1}\exp(-t)dt$ is the Euler function.
\end{lema}
\begin{proof}
The first estimate follows from integration by parts, the estimates $\bar{\P}_{x^\ast}(U_n(F) \ge t) \le2\exp(-(t/\cca)^{\alpha})$ and the formula for the expectation of Weibull variables. The second one is analogous, one simply needs to note that
the inequality $\bar{\P}_{x^\ast}(W_n(F) \ge t) \le 2\pi^\ast(\theta)^{-1}\exp(-(t/\ccb)^\alpha)$ (obtained as in the proof of Theorem \ref{thm2}) implies that
\begin{displaymath}
\bar{\P}_{x^\ast}(W_n(F) \ge t) \le \exp\Big(1 - \frac{t^\alpha}{2\log(e\pi^\ast(\theta)^{-1})\ccb^\alpha}\Big).
\end{displaymath}
\end{proof}

\begin{proof}[Proof of Theorem \ref{Markov_emp_thm}]
As in the case of a single additive functional, we decompose
\begin{displaymath}
Z := \sup_{f\in \mathcal{F}} |\sum_{i=1}^{n-1} f(X_i)| \le |U_n(F)| + \sup_{f\in \mathcal{F}}|V_n(f)| + |W_n(F)|.
\end{displaymath}

Similarly as in the proof of Theorem \ref{thm2} we get
\begin{align}\label{UF}
\bar{\P}_{x^\ast}(|U_n(F)| \ge \varepsilon t/2) \le 2\exp\Big(-\frac{(\varepsilon t)^\alpha}{(2\cca)^\alpha}\Big)
\end{align}
and
\begin{align}\label{WF}
\bar{\P}_{x^\ast}(|W_n(F)| \ge \varepsilon t/2) \le 2\pi^\ast(\theta)^{-1}\exp\Big(-\frac{(\varepsilon t)^\alpha}{(2\ccb)^\alpha}\Big).
\end{align}
Let us also note that by Lemma \ref{UW_est}

\begin{displaymath}
\bar{\E}_{x^\ast}\sup_{f\in\mathcal{F}} |\sum_{i=0}^{n-1} f(X_i)| \ge \bar{\E}_{x^\ast}\sup_{f\in\mathcal{F}} |V_n(f)| - 2\Gamma(1+1/\alpha) \cca
- 2^{1/\alpha}e\Gamma(1+1/\alpha) \log^{1/\alpha}(e/\pi^\ast(\theta))\ccb.
\end{displaymath}
By Corollary \ref{expectation_estimate_cor} we get
\begin{align*}
(1+4\varepsilon)\bar{\E}_{x^\ast}\sup_{f\in \mathcal{F}}|V_n(f)| \ge \bar{\E}_{x^\ast} \Big\|\sum_{i=1}^{(1+\varepsilon)\pi^\ast(\theta)n} s_{i-1}(f)\Big\|_{\mathcal{F}} - \Big(6 + 6{\bf e}
+ 18\pi^\ast(\theta){\bf d}^2\varepsilon^{-1}\Big)\pi(F),
\end{align*}
which together with the previous estimate gives
\begin{align*}
(1+4\varepsilon)\bar{\E}_{x^\ast} Z \ge& \bar{\E}_{x^\ast} \Big\|\sum_{i=1}^{(1+\varepsilon)\pi^\ast(\theta)n} s_{i-1}(f)\Big\|_{\mathcal{F}} - \Big(6 + 6{\bf e}
+ 18\pi^\ast(\theta){\bf d}^2\varepsilon^{-1}\Big)\pi(F) \\
&- 6\Gamma(1+1/\alpha) \cca
- 3\cdot 2^{1/\alpha}e\Gamma(1+1/\alpha) \log^{1/\alpha}(e/\pi^\ast(\theta))\ccb.
\end{align*}

When combined with the trivial bound $(1+4\varepsilon)(1+\varepsilon) \le (1+7\varepsilon)$ for $\varepsilon\in(0,1/2)$, this gives
\begin{align*}
A &:= \Big\{\sup_{f\in \mathcal{F}} |V_n(f)| \ge (1+7\varepsilon) \bar{\E}_{x^\ast}Z + (1-\varepsilon)t\Big\}\\
&\subseteq \Big\{\sup_{f\in \mathcal{F}} |V_n(f)| \ge (1+\varepsilon)\bar{\E}_{x^\ast} \Big\|\sum_{i=1}^{(1+\varepsilon)\pi^\ast(\theta)n} s_{i-1}(f)\Big\|_{\mathcal{F}} - \varepsilon C(\varepsilon)+(1-\varepsilon)t\Big\}\\
&\subseteq
\Big\{\sup_{f\in \mathcal{F}} |V_n(f)| \ge (1+\varepsilon)\bar{\E}_{x^\ast} \Big\|\sum_{i=1}^{(1+\varepsilon)\pi^\ast(\theta)n} s_{i-1}(f)\Big\|_{\mathcal{F}} + t(1-2\varepsilon)\Big\}
\end{align*}
for $t \ge C(\varepsilon)$.

By Lemma \ref{N_integr} (applied with $\varepsilon/2$ instead of $\varepsilon$) we get $$\bar{\P}_{x^\ast}(N > (1+\varepsilon) \pi^\ast(\theta) n) \le e\exp(-\varepsilon^2 n/(144\pi^\ast(\theta){\bf d}^2)),$$
which gives
\begin{align}\label{n_summand}
\bar{\P}_{x^\ast}(A) \le \bar{\P}_{x^\ast}(A\;\& N\le (1+\varepsilon)\pi^\ast(\theta)n) + e\exp(-\varepsilon^2 n/(144\pi^\ast(\theta){\bf d}^2)).
\end{align}
Now Proposition \ref{unbounded_emp_ind} gives
\begin{align*}
\bar{\P}_{x^\ast}(A&\;\& N\le (1+\varepsilon)\pi^\ast(\theta)n) \\
&\le \bar{\P}_{x^\ast}(\max_{k \le \lfloor (1+\varepsilon)\pi^\ast(\theta)n\rfloor}
\sup_{f\in \mathcal{F}}|\sum_{i=1}^k s_{i-1}(f)| \ge (1+\varepsilon)\E \Big\|\sum_{i=1}^{\lfloor(1+\varepsilon)\pi^\ast(\theta)n\rfloor} s_{i-1}(f)\Big\|_{\mathcal{F}} + t(1-2\varepsilon))\\
&\le \exp\Big(-\frac{(1-2\varepsilon)^4 t^2}{2(1+\varepsilon)^2n\sigma^2}\Big) +
e\exp\Big(-\frac{t(1-2\varepsilon)^2}{2M(1+\varepsilon^{-1})(3+4\varepsilon^{-1})}\Big) + e^8\exp\Big({-\frac{(\varepsilon (1-2\varepsilon)t)^\alpha}{2c^\alpha}}\Big),
\end{align*}
which in combination with (\ref{UF}), (\ref{WF}) and (\ref{n_summand}) ends the proof.
\end{proof}

\paragraph{Example}
The difficulty one has to deal with when applying bounds in the spirit of Theorem \ref{Markov_emp_thm} in concrete situations is controlling the order of $\E Z$. In general this is a difficult problem related to the geometry of the class $\mathcal{F}$. To illustrate the theorem we will consider a very special situation, namely, when $\mathcal{F}$ is a dense set in the unit ball of some Hilbert space, i.e. we will consider a function $G \colon \mathcal{X} \to H$ for a separable Hilbert space $(H,\|\cdot\|)$. The choice of the Hilbert space in our example is motivated by the possibility of using the parallelogram identity to estimate $\E Z^2$, however Hilbert space valued random variables appear commonly in probability and statistics, see e.g. Proposition 2 in \cite{FM} for a development related to Markov chains (the Hilbert space considered there is finite dimensional, but the constants in the inequalities do not depend on the dimension, so it fits our setting of separable Hilbert spaces).

Assume that $\pi G = 0$. We are interested in tail estimates for
\begin{displaymath}
Z = \|\sum_{i=0}^{n-1} G(X_i)\| = \sup_{f \in \mathcal{F}} |\sum_{i=0}^{n-1} f(G(X_i))|.
\end{displaymath}
Set $F = \|G\|$ and assume that the chain is geometrically ergodic and strongly aperiodic and that the parameters $\cca,\ccb,\ccc$ for $F$ are finite, which is e.g. the case if $F$ is controlled in terms of the drift function, like in Theorem \ref{thm:B_new}.
Recall \eqref{decomp} and \eqref{N_def}, which imply that
\begin{align*}
\E_x Z &\le \bar{\E}_{x^\ast} U_n(F) + \bar{\E}_{x^\ast} W_n(F) + \E \max_{k\le n-1} \|\sum_{i=1}^k s_{i-1}(G)\|\\
&\le  2\Gamma(1+1/\alpha) \cca + 2^{1/\alpha}e\Gamma(1+1/\alpha) \log^{1/\alpha}(e/\pi^\ast(\theta))\ccb
+ 2(\E\|\sum_{i=1}^{n-1} s_{i-1}(G)\|^2)^{1/2},
\end{align*}
where we used Lemma \eqref{UW_est} to handle $\bar{\E}_{x^\ast} U_n(F)$ and $\bar{\E}_{x^\ast} W_n(F)$ and Doob's inequality to deal with the last summand. Now, due to the parallelogram identity, independence and centeredness of the variables $s_i$  and \eqref{eq:boundonsigma} we have
\begin{displaymath}
\E\|\sum_{i=1}^{n-1} s_{i-1}(G)\|^2 = n \E \|s_0(G)\|^2 \le n \E s_0(F)^2 \le 4\alpha^{-1}\Gamma(2/\alpha) \ccc^2  n.
\end{displaymath}
Thus we get
\begin{displaymath}
\E_x Z \le 2\Gamma(1+1/\alpha) \cca + 2^{1/\alpha}e\Gamma(1+1/\alpha) \log^{1/\alpha}(e/\pi^\ast(\theta))\ccb +
4\alpha^{-1/2}\Gamma(2/\alpha)^{1/2} \ccc \sqrt{n}.
\end{displaymath}

Thus Theorem \ref{Markov_emp_thm} yields
\begin{align*}
&\P_x\Big(\|\frac{1}{n}\sum_{i=0}^{n-1} G(X_i)\| \ge \\
&\phantom{aaaaa}(1+7\varepsilon)\Big(\frac{2\Gamma(1+1/\alpha) \cca+2^{1/\alpha}e\Gamma(1+1/\alpha) \log^{1/\alpha}(e/\pi^\ast(\theta))\ccb}{n} +
\frac{4\alpha^{-1/2}\Gamma(2/\alpha)^{1/2} \ccc}{\sqrt{n}} \Big) + t\Big) \\
&\le
\exp\Big(-\frac{(1-2\varepsilon)^4 n t^2}{2(1+\varepsilon)^2 \sigma^2}\Big) +
e\exp\Big(-\frac{nt(1-2\varepsilon)^2}{2M(1+\varepsilon^{-1})(3+4\varepsilon^{-1})}\Big)
+ e^8\exp\Big({-\frac{(\varepsilon (1-2\varepsilon)n t)^\alpha}{2\ccc^\alpha}}\Big)\\
&+e\exp\Big(-\frac{\varepsilon^2 n}{144{\pi^\ast(\theta)\bf d}^2}\Big)
 + 2\exp\Big(-\frac{(\varepsilon n t)^\alpha}{(2\cca)^\alpha}\Big)
+2\pi^\ast(\theta)^{-1}\exp\Big(-\frac{(\varepsilon n t)^\alpha}{(2\ccb)^\alpha}\Big)
\end{align*}
where $M=\ccc(3\alpha^{-2}\log(n))^{1/\alpha}$, provided that $nt \ge C(\varepsilon)$. In particular for $n$ large enough, depending explicitly on $t$ and $\varepsilon$ we get the same bound on
\begin{displaymath}
\P_x\Big(\|\frac{1}{n}\sum_{i=0}^{n-1} G(X_i)\| \ge 2t\Big).
\end{displaymath}

We note that with this Law of Large Numbers normalization, where one is usually interested in $t$ being a small constant, independent of $n$ (say $t \in (0,1)$), the fourth summand on the right hand side above, which is independent of $t$, does not dominate the whole sum.

Let us also remark that estimates of the same order for the expectation could be also achieved under some bracketing or VC-dimension assumption on the class $\mathcal{F}$. This is well known in the independent case (see e.g. the monograph \cite{vdVW}) and there are some results concerning Markov chains or mixing sequences (e.g. \cite{BL,Levental1,Rio2}).

\noindent University of Warsaw\\
Banacha 2\\
02-097 \\
Warszawa\\
Poland\\
E-mail: R.Adamczak@mimuw.edu.pl, W.Bednorz@mimuw.edu.pl
\end{document}